\documentclass{ws-m3as}

\makeatletter
\def\catchline#1#2#3#4#5{\relax} 
\let\trimmarks\relax 
\makeatother

\usepackage{amsmath,
            amssymb,
            mathtools,
            bm,
            mathrsfs,
            stmaryrd,
            upgreek,
            textcomp,
            cancel,
            stackengine,
            eqlist,
            graphicx,
            subcaption
}

\usepackage[np,autolanguage]{numprint}
\usepackage{
            multirow,
            float,
            rotating,
            afterpage,
            url%
}

\usepackage{nth}

\usepackage{ifthen}
\usepackage{calc}
\usepackage{xspace}

\usepackage{comment}

\usepackage{tikz}
\newcommand*\circd[1]{%
  \begin{tikzpicture}
    \node[draw,circle,inner sep=0.5pt]{\scriptsize #1};
  \end{tikzpicture}}

\usepackage{acronym}
\acrodef{CRF}[CRF]{constitutive response function}
\acrodef{SI}[SI]{International System of Units}
\acrodef{BVP}[BVP]{boundary value problem}
\acrodef{IBVP}[IBVP]{initial-boundary value problem}
\acrodef{FEM}[FEM]{finite element method}
\acrodef{EIT}[EIT]{extension, inflation, and torsion}
\acrodef{ALE}[ALE]{arbitrary {L}agrangian--{E}ulerian}

\acrodef{BDF}[BDF]{backward differentiation formula}
\acrodef{CNS}[CNS]{central nervous system}
\acrodef{CSF}[CSF]{cerebrospinal fluid}
\acrodef{IF}[IF]{interstitial fluid}
\acrodef{ECS}[ECS]{extracellular space}
\acrodef{SAS}[SAS]{subarachnoid space}
\acrodef{PVS}[PVS]{perivascular spaces}
\acrodef{TPE}[TPE]{theorem of power expended}

\def\etal{\emph{et al.\/}}

\newcommand{\scdot}{\ensuremath{\!\cdot\!}}
\newcommand{\Grad}[2][nonsense]{%
\ifthenelse{\equal{#1}{nonsense}}%
{\nabla#2}%
{\nabla_{#1}#2}%
}
\newcommand{\Div}[2][nonsense]{%
\ifthenelse{\equal{#1}{nonsense}}%
{\nabla\scdot#2}%
{\nabla_{#1}\scdot#2}%
}
\newcommand{\Lin}[2][nonsense]{%
\ifthenelse{\equal{#1}{nonsense}}%
{\ensuremath{\mathscr{L}(#2)}}%
{\ensuremath{\mathscr{L}(#1,#2)}}%
}

\DeclareMathOperator{\sym}{sym}

\DeclareMathOperator{\grad}{grad}

\renewcommand{\d}[2][nonsense]{%
\ifthenelse{\equal{#1}{nonsense}}%
{\ensuremath{\mathrm{d}#2}}%
{\ensuremath{\mathrm{d}^{#1}#2}}
}
\newcommand{\bv}[1]{\boldsymbol{\mathbf{#1}}}

\newcommand{\ts}[1]{\text{$\mathsf{#1}$}}
\newcommand{\transpose}[1]{#1^{\scriptscriptstyle\mathrm{T}}}
\newcommand{\inversetranspose}[1]{\ensuremath{#1^{\scriptscriptstyle-\mathrm{T}}}}
\newcommand{\order}[3][nonsense]{%
\ifthenelse{\equal{#1}{nonsense}}%
{\ensuremath{#2^{\scriptscriptstyle(#3)}}}%
{\ensuremath{#2^{\scriptscriptstyle#1(#3)}}}%
}

\newcommand{\colondot}{\mathbin{:}}

\newcommand{\s}{\ensuremath{\mathrm{s}}}
\newcommand{\f}{\ensuremath{\mathrm{f}}}
\newcommand{\flt}{\ensuremath{\mathrm{flt}}}
\newcommand{\frs}{\ensuremath{\mathrm{sf}}}
\newcommand{\va}{\ensuremath{\mathrm{va}}}
\newcommand{\ip}{\mathfrak{p}}

\newcommand{\comsol}{COMSOL Multiphysics\textsuperscript{\textregistered}\xspace}

\begin{document}

\markboth{Francesco Costanzo, Mohammad Jannesari, and Beatrice Ghitti}{Poroelastic flow across a permeable interface}

%
\catchline{}{}{}{}{}
%

\title{Poroelastic flow across a permeable interface:
a Hamilton's principle approach and its finite element implementation}

\author{Francesco Costanzo\footnote{Corresponding author.}}
\address{Center for Neural Engineering, University Park, PA 16802, USA\\
Engineering Science and Mechanics Department, University Park, PA 16802, USA\\fxc8@psu.edu}

\author{Mohammad Jannesari}
\address{Center for Neural Engineering, University Park, PA 16802, USA\\
Engineering Science and Mechanics Department, University Park, PA 16802, USA\\mbj5423@psu.edu}

\author{Beatrice Ghitti}
\address{Center for Neural Engineering, University Park, PA 16802, USA\\
Engineering Science and Mechanics Department, University Park, PA 16802, USA\\
Auckland Bioengineering Institute, The University of Auckland, Auckland 1010, New Zealand\\beatrice.ghitti@auckland.ac.nz}




\maketitle

\begin{abstract}
    We consider fluid flow across a permeable interface within a deformable porous medium. We use mixture theory. The mixture's constituents are assumed to be incompressible in their pure form. We use Hamilton's principle to obtain the governing equations, and we propose a corresponding finite element implementation. The filtration velocity and the pore pressure are allowed to be discontinuous across the interface while some control of these discontinuities is built into the interfacial constitutive behavior. To facilitate the practical implementation of the formulation in a finite element scheme, we introduce a Lagrange multiplier field over the interface for the explicit enforcement of the jump condition of the balance of mass. Our formulation appears to recover some basic results from the literature. The novelty of the work is the formulation of an approach that can accommodate specific constitutive assumptions pertaining to the behavior of the interface that do not necessarily imply the continuity of the filtration velocity and/or of the pore pressure across it.
\end{abstract}

\keywords{Poroelasticity; Jump Conditions; Hamilton's Principle; Finite Element Method.}

\ccode{AMS Subject Classification: 74A99, 76S05, 74S05, 76-10}


\allowdisplaybreaks{                             %

\section{Introduction}
This paper is motivated by our interest in modeling aspects of brain physiology. Poroelasticity can be effective in modeling \ac{IF} flow through brain parenchyma and its interaction with the \ac{CSF} circulation. While this paper is not about brain physiology, we briefly touch upon the problem of identifying the pathways and drivers of waste clearance in the brain to provide some context for what we present herein.

The accumulation of potentially neurotoxic waste products, such as amyloid-$\beta$, in the brain \ac{ECS} has been linked to the pathogenesis of neurodegenerative disorders like Alzheimer's disease.\cite{Selkoe2016The-amyloid-hyp}  The brain does not have a conventional lymphatic system\cite{Abbott2004Evidence-for-bu} responsible for clearing products of metabolism and neurotransmission,\cite{Nedergaard2013Garbage-Truck-o,Rasmussen2022Fluid-transport} and the identification of clearance pathways and driving mechanisms remains a poorly understood aspect of brain physiology and pathology.\cite{Hladky2018Elimination-of-,Hladky2022The-glymphatic-,Abbott2018The-role-of-bra}

The brain and spinal cord are surrounded by \ac{CSF}, which flows through a system of fluid-filled communicating spaces including the \ac{SAS}, ventricular system, central canal of the spinal cord, and the \ac{PVS} surrounding cerebral vessels. Transport of waste by \ac{CSF} flow is acknowledged as a key process in maintaining brain function. Transport of metabolites through parenchyma has traditionally been attributed to Fickian diffusion\cite{Sykova2004Diffusion-prope,Sykova2008Diffusion-in-br,Vargova2011Glia-and-extrac,Verkman2013Diffusion-in-th} eventually leading to mixing with \ac{CSF} and draining into the peripheral lymphatics in the neck.\cite{Bradbury1981Drainage-of-cer,Bradbury1983Factors-influen} More recently, the work by Iliff, Nedergaard, and co-workers\cite{Iliff2012A-Paravascular-,Iliff2013Brain--wide-pat,Iliff2013Cerebral-arteri,Iliff2019The-glymphatic-} has challenged this understanding, pointing to the existence of a directional flow of fluid into and out of the \ac{CNS}. This \emph{circulation}  has been dubbed the \emph{glymphatic system} and is characterized by a significant convective transport and exchange between \ac{IF} and \ac{CSF} along the \ac{PVS}. However, several aspects of the glymphatic hypothesis remain controversial\cite{Hladky2018Elimination-of-,Hladky2022The-glymphatic-,Abbott2018The-role-of-bra}: little is known about the mechanisms and forces driving fluid movement. Various drivers of the glymphatic flow have been proposed in the literature, such as arterial pulsation,\cite{Iliff2013Cerebral-arteri,Mestre2018Flow-of-cerebro,Bedussi2018Paravascular-sp} functional hyperemia,\cite{Kedarasetti2020Functional-hype,Kedarasetti2022Arterial-vasodi} respiration,\cite{Kiviniemi2016Ultra-fast-magn} and body posture.\cite{Lee2015The-Effect-of-B}

The transport phenomena at hand are characterized by flow through parenchyma that has both convective and diffusive features. Furthermore, this flow can cross a variety of permeable membranes, such as the glia limitans and the pia mater (cf.\ Ref.~\refcite{Hladky2022The-glymphatic-}). These membranes have different cellular structures and can play different roles, including that of regulating  osmotic pressures. To simultaneously account for both diffusive and convective flows, we had adopted\cite{Costanzo2016Finite-Element-0,Kedarasetti2022Arterial-vasodi} a poroelasticity theory based on mixture theory.\cite{Bowen1976Theory-of-Mixtu0,Bowen1980Incompressible-0} In this paper, we focus on interface conditions at permeable interfaces with an eye toward the possibility of explicitly reflecting interfacial properties in our model.

There is a vast literature on the conditions at the interface between a porous medium and a fluid (see, e.g., in Refs.~\refcite{Hou1989Boundary-condit,Shim2022A-Hybrid-Biphas}, and~\refcite{dellIsola2009Boundary-Condit-0}). This literature concerns both modeling aspects (e.g., Refs.~\refcite{Hou1989Boundary-condit,Shim2022A-Hybrid-Biphas}, and~\refcite{dellIsola2009Boundary-Condit-0}) as well as the numerical implementation of such conditions (e.g., Refs.~\refcite{Lee2019A-Mixed-Finite-,Causemann2022Human-Intracran}, and~\refcite{Ruiz-Baier2022The-Biot-Stokes}). It is well-known that the jump conditions of the balance laws do not imply or demand that the fluid velocity field or the pore pressure be continuous across the interface.\cite{Bowen1976Theory-of-Mixtu0} It is also well-known that a \ac{BVP} without sufficient constraints on said discontinuities is generally ill-posed (e.g., see discussion in Ref.~\refcite{Hou1989Boundary-condit}) and such that the ill-posedness has different manifestations depending on the fluid's constitutive behavior on the two sides of the interface. Therefore, complementing the jump conditions of the balance laws with some additional interface conditions is often essential for the formulation of a well-posed \ac{BVP}. Provided that several approaches have been proposed in the literature, by and large, the most commonly adopted interface conditions (cf.\ Ref.~\refcite{Shim2022A-Hybrid-Biphas}) are kinematic in nature as they prescribe continuity for the pore pressure field and the tangential component of the filtration velocity. For mixtures of incompressible materials, the continuity of the tangential component of the filtration velocity, along with the requirement that its normal component be continuous as a consequence of the balance of mass, yields a continuity condition on the filtration velocity as a whole, leaving the fluid velocity field discontinuous. Such continuity requirement is fairly straightforward to implement in practice whenever the field equations are written in terms of filtration velocity, as opposed to the fluid velocity. However, this type of formulation may not always be appropriate for the modeling of a particular problem. This said, and much more importantly, prescribing the continuity of certain fields across the interface is tantamount to prescribing an \emph{implicit} constitutive response to the interface. This idea is brought out by the analysis in Ref.~\refcite{dellIsola2009Boundary-Condit-0} where the authors use Hamilton's principle\cite{Bedford2021Hamiltons-Princ} to derive the problem's governing equations and interface conditions. In this approach, one must define an action potential and, when dissipation is present,  a Rayleigh pseudo-potential. In turn, this requires an explicit model of all energy storage and dissipation mechanisms in the system. As such, the interface constitutive behavior is explicitly modeled and the associated interface conditions are derived by a choice of test functions that appropriately respects the kinematic assumptions underlying the model. In doing so, dell'Isola \etal{}\cite{dellIsola2009Boundary-Condit-0} obtained interface conditions that had not previously appeared in the literature but that transparently translated the constitutive behavior of the interface and its relation to the constitutive behavior of the surrounding materials, regardless of whether they are both poroelastic, or one being poroelastic and the other being pure fluid.

We found the work presented in Ref.~\refcite{dellIsola2009Boundary-Condit-0} to provide the most appropriate setting for our interest. This said, we could not find a numerical implementation of such framework, especially for the case in which the constituents are incompressible in their pure form. Furthermore, to obtain the interface conditions pertaining to the motion of the fluid relative to the solid, the authors of Ref.~\refcite{dellIsola2009Boundary-Condit-0} designed two \emph{ad hoc} test functions for each of two different components of the interface conditions and distinguishing the cases with and without mass transport across the interface. While we find said construction rather clever, we also find it somewhat difficult to implement it practically within common \ac{FEM} platforms. With this in mind, the objectives of this work are (\emph{i}) to  propose an adaptation of the variational approach in Ref.~\refcite{dellIsola2009Boundary-Condit-0} to the case in which each of the mixture constituents is incompressible in its pure form, and (\emph{ii}) to propose a practical implementation of the approach using a \ac{FEM} that allows the use of commonly available discontinuous representations of fluid velocity fields.

The paper is organized as follows. First, in Sec.~\ref{sec: Basics}, we present the essential notation and basic kinematics needed to describe our model. In Sec.~\ref{Sec: Balance laws}, we then discuss the relevant balance laws following the traditional approach in Ref.~\refcite{Bowen1976Theory-of-Mixtu0} and we derive a corresponding energy estimate. In Sec.~\ref{sec: variational formulations}, we  derive the field equations using Hamilton's principle and offer a comparison with the previously derived governing equations. In this section we also propose a practical \ac{FEM} implementation of our equations. In Sec.~\ref{Sec: Examples}, we discuss some examples. In Sec.~\ref{Sec: Summary}, we conclude the paper with a summary and discussion.

\section{Basic definitions and kinematics}
\label{sec: Basics}
Here we describe the motion of our system. Our treatment is rooted in Refs.~\refcite{Bowen1976Theory-of-Mixtu0} and~\refcite{Bowen1980Incompressible-0} and we adopt the kinematic and analytical framework in Ref.~\refcite{dellIsola2009Boundary-Condit-0}, which we adapt to the case of a biphasic mixture of incompressible constituents. For completeness, basic definitions have been placed in Appendix~\ref{Appendix: Notation}.

\subsection{The physical system and its kinematics}
\label{subsec: kinematics}
We consider a mixture with a solid and a fluid phase. The mixture's current configuration is $B(t)\subset\mathcal{E}^{3}$, while $B_{\s} \subset \mathcal{E}^{3}$ and $B_{\f} \subset \mathcal{E}^{3}$ are the reference configurations of the solid and fluid phases, respectively. The symbols $\bv{x}$, $\bv{X}_{\s}$, and $\bv{X}_{\f}$ denote points in $B(t)$, $B_{\s}$, and $B_{\f}$, respectively. Referring to the notation in Eq.~\eqref{eq: circd notation}, $\circd{e}$, $\circd{s}$, and $\circd{f}$ indicate $B(t)$, $B_{\s}$, and $B_{\f}$, respectively.

$B(t)$ is the common image of the solid and fluid motions, respectively
\begin{equation}
\label{eq: Motions defs}
\bv{\chi}_{\s}: B_{\s} \times \mathbb{R}_{0}^{+}\to B(t)
\quad\text{and}\quad
\bv{\chi}_{\f}: B_{\f} \times \mathbb{R}_{0}^{+}\to B(t),
\end{equation}
which are expected to be diffeomorphisms except, possibly, across a permeable interface $\mathcal{S}(t) \subset B(t)$ separating two (or more) porous domains and/or a porous and a purely fluid domain. 

If $\mathcal{S}(t)$ separates two distinct porous domains, $\mathcal{S}(t)$ is assumed to be a coherent interface convected by a globally continuous solid motion. If $\mathcal{S}(t)$ separates a porous domain from a purely fluid domain, $\bv{\chi}_{\s}$ is understood to be continuously extended in the fluid domain. In all cases, $\mathcal{S}(t)$ is the image under $\bv{\chi}_{\s}$ of a smooth not self-intersecting \emph{fixed} submanifold $\mathcal{S}_{\s}\subset B_{\s}$ such that $\mathcal{S}_{\s}$ either cuts through $B_{\s}$ intersecting the boundary of $B_{\s}$ or is closed and completely contained in $B_{\s}$. $\mathcal{S}_{\f}(t)$ is the inverse image of $\mathcal{S}(t)$ under $\bv{\chi}_{\f}$.

Next, away from $\mathcal{S}_{\s}$ and $\mathcal{S}_{\f}(t)$, the gradients (the meaning of $\nabla$ is context dependent, cf.\ Appendix~\ref{Appendix: Notation}) of $\bv{\chi}_{\s}$ and $\bv{\chi}_{\f}$ are
\begin{equation}
\label{eq: Fs and Ff defs}
    \ts{F}_{\s}(\bv{X}_{\s},t) \coloneqq \Grad{\bv{\chi}_{\s}(\bv{X}_{\s},t)}
\quad\text{and}\quad    
    \ts{F}_{\f}(\bv{X}_{\f},t) \coloneqq \Grad{\bv{\chi}_{\f}(\bv{X}_{\f},t)},
\end{equation}
with determinants $J_{\s}$ and $J_{\f}$, respectively, which are taken to be strictly positive. Similarly, the solid and fluid material velocity fields are, respectively, 
\begin{equation}
\label{eq: vs and vf defs}
    \bv{v}_{\s}(\bv{X}_{\s},t) \coloneqq \partial_{t}\bv{\chi}_{\s}(\bv{X}_{\s},t)
    \quad\text{and}\quad
    \bv{v}_{\f}(\bv{X}_{\f},t) \coloneqq \partial_{t}\bv{\chi}_{\f}(\bv{X}_{\f},t).
\end{equation}
In Ref.~\refcite{dellIsola2009Boundary-Condit-0}, it is posited that the motions $\bv{\chi}_{\s}$ and $\bv{\chi}_{\f}$ are diffeomorphisms away from the interface, but can suffer jump discontinuities in their space--time gradients so as to have the Hadamard properties.\cite{dellIsola2009Boundary-Condit-0} A general discussion of the subject can be found in Refs.~\refcite{dellIsola2009Boundary-Condit-0,Truesdell1965The-Non-Linear-0}, and~\refcite{Bowen1976Theory-of-Mixtu0}. Here, we start with following the same kinematic assumption for $\bv{\chi}_{\s}$ as in Ref.~\refcite{dellIsola2009Boundary-Condit-0}, so that $\bv{\chi}_{\s}$ is continuous across the interface. We will also assume that, away from the interface, $\bv{\chi}_{\f}$ is a diffeomorphism. However, we will allow the fluid material velocity field to be discontinuous across the interface while only restricted to the requirements imposed by the jump condition of the mass balance. The kinematic constraint originating with the latter will be imposed via a Lagrange multiplier in Sec.~\ref{subsec: Hamilton principle}.


We denote by  $\bv{m}_{\s}$ the unit normal orienting $\mathcal{S}_{\s}$. The corresponding normal orienting $\mathcal{S}(t)$ is $\bv{m}$, which is given by
\begin{equation}
\label{eq: normal to St}
\bv{m} \coloneqq \biggl(\frac{J_{\s}\inversetranspose{\ts{F}}_{\s}\bv{m}_{\s}}{\|\inversetranspose{J_{\s}\ts{F}}_{\s} \bv{m}_{\s}\|}\biggr)^{\circd{e}},
\end{equation}
and the normal velocity of $\mathcal{S}(t)$ is the unique scalar field
\begin{equation}
\label{eq: celerity of St}
c_{\mathcal{S}(t)}: \mathcal{S}(t)\to\mathbb{R},
\quad
c_{\mathcal{S}(t)} \coloneqq \bv{v}_{\s}^{\circd{e}}\cdot\bv{m}.
\end{equation}

In addition to the maps $\bv{\chi}_{\s}$ and $\bv{\chi}_{\f}$, following Ref.~\refcite{dellIsola2009Boundary-Condit-0}, we introduce the following map
\begin{equation}
\label{eq: Relative motion}
\bv{\chi}_{\frs}: B_{\s}\to B_{\f},\quad
\bv{\chi}_{\frs} \coloneqq \bv{\chi}_{\f}^{-1}\circ\bv{\chi}_{\s}.
\end{equation}
The function $\bv{\chi}_{\frs}(\bv{X}_{\s},t)$, which is well-defined away from the interface, identifies the fluid particle $\bv{X}_{\f}$ that, at time $t$, occupies the same position in the mixture as the solid material particle $\bv{X}_{\s}$. This map encodes the relative motion of the fluid with respect to the solid. Letting 
\begin{equation}
\label{eq: Ffrs and vfrs defs}
\ts{F}_{\frs} \coloneqq \Grad\bv{\chi}_{\frs},\quad
J_{\frs} = \det\ts{F}_{\frs},
\quad\text{and}\quad
\bv{v}_{\frs} \coloneqq \partial_{t}\bv{\chi}_{\frs},
\end{equation}
it can be shown that
\begin{equation}
\label{eq: vf rel vs}
\ts{F}_{\f}^{\circd{s}} = \ts{F}_{\s}\ts{F}_{\frs}^{-1},\quad
J_{\f}^{\circd{s}} = J_{\s}/J_{\frs},
\quad\text{and}\quad
\bv{v}_{\f}^{\circd{\s}} = \bv{v}_{\s} - \ts{F}_{\s}\ts{F}_{\frs}^{-1} \bv{v}_{\frs}.
\end{equation}

If the solid and fluid are submanifolds of an all-containing Euclidean manifold, the displacement of the solid and fluid phases, denoted by $\bv{u}_{\s}$ and $\bv{u}_{\f}$, respectively, are defined as
\begin{align}
\label{eq: us and uf defs}
    \bv{u}_{\s}(\bv{X}_{\s},t) \coloneqq \bv{\chi}_{\s}(\bv{X}_{\s},t) - \bv{X}_{\s}
    \quad\text{and}\quad
    \bv{u}_{\f}(\bv{X}_{\f},t) \coloneqq \bv{\chi}_{\f}(\bv{X}_{\f},t) - \bv{X}_{\f}.
\end{align}
We note that
\begin{equation}
\label{eq: uf from relative motion}
\bv{u}_{\f}^{\circd{s}} = \bv{\chi}_{\s} - \bv{\chi}_{\frs}.
\end{equation}
The material accelerations of the solid and fluid phases, denoted by $\bv{a}_{\s}$ and $\bv{a}_{\f}$, respectively, are
\begin{align}
\label{eq: as and af defs}
\bv{a}_{\s}(\bv{X}_{\s},t) \coloneqq \partial_{t}\bv{v}_{\s}(\bv{X}_{\s},t)
\quad\text{and}\quad
    \bv{a}_{\f}(\bv{X}_{\f},t) \coloneqq \partial_{t}\bv{v}_{\f}(\bv{X}_{\f},t).
\end{align}

\section{Balance of mass and momentum}
\label{Sec: Balance laws}
Here we summarize the presentation in Refs. \refcite{Bowen1976Theory-of-Mixtu0} and \refcite{Bowen1980Incompressible-0}. 

\subsection{Balance of mass}
In their pure form, the fluid and the solid phases are assumed to be incompressible with mass densities given by the constants $\rho_{\f}^{*}$ and $\rho_{\s}^{*}$, respectively.

Let
\begin{equation}
\label{eq: volume fractions}
\phi_{\s}: B(t)\to[0,1]
\quad\text{and}\quad
\phi_{\f}: B(t)\to[0,1]
\end{equation}
denote the volume fractions of the solid and fluid, respectively, subject to the \emph{saturation condition}
\begin{equation}
\label{eq: saturation condition}
\phi_{\s}+\phi_{\f} = 1.
\end{equation}
Then, the mass densities of the solid and the fluid over $B(t)$ (i.e., in the mixture) are, respectively,
\begin{equation}
\label{eq: mass density distributions}
\rho_{\s} = \rho_{\s}^{*} \phi_{\s}
\quad\text{and}\quad
\rho_{\f} = \rho_{\f}^{*} \phi_{\f}.
\end{equation}
In the absence of chemical reactions, the mass balance in $B(t)$ requires that
\begin{equation}
\label{eq: balance of mass}
\partial_{t}\rho_{a} + \Div{(\rho_{a} \bv{v}_{a}^{\circd{e}})} = 0
\quad\Rightarrow\quad
\partial_{t}\phi_{a} + \Div{(\phi_{a} \bv{v}_{a}^{\circd{e}})} = 0,
\end{equation}
for each species $a = \s,\f$.
In view of Eq.~\eqref{eq: saturation condition}, the motion of the mixture is governed by the following (single) constraint:
\begin{equation}
\label{eq: div vvavg constraint}
\Div{(\phi_{\s}\bv{v}_{\s}^{\circd{e}}+\phi_{\f}\bv{v}_{\f}^{\circd{e}})} = 0.
\end{equation}
We now define the following fields over $B(t)$:
\begin{equation}
\label{eq: vvavg and vflt defs}
\bv{v}_{\va} \coloneqq \phi_{\s}\bv{v}_{\s}^{\circd{e}}+\phi_{\f}\bv{v}_{\f}^{\circd{e}}
\quad\text{and}\quad
\bv{v}_{\flt} \coloneqq \phi_{\f}(\bv{v}_{\f}^{\circd{e}} - \bv{v}_{\s}^{\circd{e}}).
\end{equation}
The field $\bv{v}_{\va}$ is the volume average velocity of the mixture, whereas $\bv{v}_{\flt}$ is the filtration velocity. Using Eqs.~\eqref{eq: vvavg and vflt defs}, Eq.~\eqref{eq: div vvavg constraint} can also be written as
\begin{equation}
\label{eq: div vvavg constraint expanded}
\Div{\bv{v}_{\va}} = 0
\quad\text{or}\quad
\Div{\bv{v}_{\s}^{\circd{e}} + \Div\bv{v}_{\flt}} = 0.
\end{equation}

As $\bv{v}_{\f}$ may not be continuous across $\mathcal{S}(t)$, we also need to state the jump condition of the balance of mass. We assume that the interface does not support the production or removal of mass. Hence, using traditional arguments (cf.\ Refs.~\refcite{Bowen1976Theory-of-Mixtu0} and~\refcite{GurtinFried_2010_The-Mechanics_0}), at every point on $\mathcal{S}(t)$ we must have
\begin{equation}
\label{eq: mass balance jump condition}
\llbracket \bv{v}_{\flt} \rrbracket \cdot \bv{m} = 0
\quad\Rightarrow\quad
\llbracket \bv{v}_{\va} \rrbracket \cdot \bv{m} = 0,
\end{equation}
where the second of Eqs.~\eqref{eq: mass balance jump condition} is a consequence of the continuity of $\bv{v}_{\s}$.

The second of Eqs.~\eqref{eq: balance of mass} and the assumed regularity of the motions away from the interface imply the existence of referential forms of the volume fraction distributions for the solid and fluid phases, which we denote by $\phi_{R_{\s}}$ and $\phi_{R_{\f}}$, respectively:
\begin{equation}
\label{eq: phiRs and phiRf defs}
\phi_{R_{\s}} = J_{\s} \phi_{\s}^{\circd{s}}
\quad\text{and}\quad
\phi_{R_{\f}} = J_{\f} \phi_{\f}^{\circd{f}}.
\end{equation}

\begin{remark}[Solid immersed in fluid]
In the present kinematic framework, it is admissible for $\phi_{R_{\s}}$ to vanish over parts of $B_{\s}$. In such cases, the support of $\phi_{R_{\s}}$ will be assumed to be a closed manifold that supports the physics of a solid body immersed in a fluid.    
\end{remark}

\subsection{Balance of momentum}
The mixture is assumed to be isothermal. Using Refs.~\refcite{Bowen1976Theory-of-Mixtu0} and~\refcite{Bowen1980Incompressible-0} as guides, the balance of momentum for phase $a = \s,\f$ in a mixture requires that in $B(t)$ away from $\mathcal{S}(t)$ we have
\begin{equation}
\label{eq: momentum}
\rho_{a} \bv{a}_{a}^{\circd{e}} - \Div{\ts{T}_{a}} -\bv{b}_{a} - \bv{p}_{a} = \bv{0},
\end{equation}
where $\ts{T}_{a}$ is the Cauchy stress of species $a$, $\bv{b}_{a}$ is an external force density (per unit volume of $B(t)$), and $\bv{p}_{a}$ is the force density (per unit volume of $B(t)$) acting on $a$ because of its interaction with the rest of the mixture. The interaction forces must satisfy the constraint
\begin{equation}
\label{eq: interaction force constraint}
\sum_{a} \bv{p}_{a} = \bv{0}.
\end{equation}
For the momentum balance jump condition over $\mathcal{S}(t)$ we posit that
\begin{equation}
\label{eq: jump of momentum}
\llbracket
\rho_{a}(\bv{v}_{a}^{\circd{e}} - \bv{v}_{\s}^{\circd{e}}) \otimes (\bv{v}_{a}^{\circd{e}} - \bv{v}_{\s}^{\circd{e}}) - \ts{T}_{a}
\rrbracket \, \bv{m}
+ \bv{\xi}_{a}
= \bv{0},
\end{equation}
where $\bv{\xi}_{a}$ is a force density per unit area of $\mathcal{S}(t)$ analogous to $\bv{p}_{a}$, and for which we posit a similar constraint
\begin{equation}
\label{eq: interface interaction force constraint}
\sum_{a} \bv{\xi}_{a} = \bv{0}.
\end{equation}
Equation~\eqref{eq: jump of momentum} differs from what can be found in Ref.~\refcite{Bowen1976Theory-of-Mixtu0} (cf.\ Eq.~(2.10.22) in Ref.~\refcite{Bowen1976Theory-of-Mixtu0}) because of $\bv{\xi}_{a}$.  In Remark~\ref{remark: on the force xa} of Section~\ref{subsec: Hamilton principle}, we argue that $\bv{\xi}_{a}$ plays a role in modeling dissipation at the interface and in the satisfaction of the first of Eqs.~\eqref{eq: div vvavg constraint expanded}.

Equations~\eqref{eq: jump of momentum} and~\eqref{eq: interface interaction force constraint} yield the following conditions for the solid, fluid, and mixture, respectively:
\begin{gather}
\label{eq: Jump condition for solid and fluids individually}
- \llbracket \ts{T}_{\s}\rrbracket \bv{m} + \bv{\xi}_{\s} = \bv{0},
\quad
d \llbracket\bv{v}_{\f}^{\circd{e}}\rrbracket - \llbracket \ts{T}_{\f} \rrbracket \bv{m} + \bv{\xi}_{\f} = \bv{0},
\shortintertext{and}
\label{eq: Jump condition for mixture}
d \llbracket\bv{v}_{\f}^{\circd{e}}\rrbracket - \llbracket \ts{T}_{\f} + \ts{T}_{\s} \rrbracket \bv{m} = \bv{0},
\end{gather}
where $d$ is
\begin{equation}
\label{eq: d def}
d \coloneqq \rho_{\f} (\bv{v}_{\f}^{\circd{e}} - \bv{v}_{\s}^{\circd{e}}) \cdot \bv{m}
\quad\Rightarrow\quad
\llbracket d \rrbracket = 0.
\end{equation}

\begin{remark}
    The jump condition in Eq.~\eqref{eq: Jump condition for mixture} is generally agreed upon in the literature (cf.\ Eq.~(11\emph{b}) in Ref.~\refcite{Hou1989Boundary-condit}, Eq.~(2.73) in Ref.~\refcite{Shim2022A-Hybrid-Biphas}, and Eq.~(30) in Ref.~\refcite{dellIsola2009Boundary-Condit-0}). However, the same cannot be said for Eqs.~\eqref{eq: Jump condition for solid and fluids individually}. The tangential projection of the second of Eqs.~\eqref{eq: Jump condition for solid and fluids individually} is similar to Eq.~(32) in Ref.~\refcite{dellIsola2009Boundary-Condit-0}.
\end{remark}

\subsection{Constitutive response functions}
\label{sec: CRF}
We assume that the pure solid phase is incompressible and hyperelastic. The strain energy of the pure solid phase is assumed to be a non-negative scalar function of the form
\begin{equation}
\label{eq: pure solid strain energy}
\hat{\Psi}_{\s} = \hat{\Psi}_{\s}(\overline{\ts{C}}_{\s},\bv{X}_{\s}),
\end{equation}
where
\begin{equation}
\label{eq: right CG strain tensor}
\ts{C}_{\s} \coloneqq \transpose{\ts{F}}_{\s} \ts{F}_{\s}
\quad\text{and}\quad
\overline{\ts{C}}_{\s} = J_{\s}^{-2/3} \ts{C}_{\s},
\end{equation}
so that $\overline{\ts{C}}_{\s}$ is the isochoric part of $\ts{C}_{\s}$. The strain energy of the solid phase in the mixture is simply assumed to that of the pure solid scaled by its volume fraction:
\begin{equation}
\label{eq: solid strain energy}
\Psi_{\s} \coloneqq \phi_{R_{\s}} \hat{\Psi}_{\s}.
\end{equation}
It is convenient to introduce a symbol to refer to the solid's strain energy per unit volume of the current configuration:
\begin{equation}
\label{eq: eulerian strain energy}
\psi_{\s} \coloneqq \bigl(J_{\s}^{-1}\Psi_{\s}\bigr)^{\circd{e}}.
\end{equation}
The contributions to the first Piola-Kirchhoff and Cauchy stress tensors originating from the elastic strain energy are, respectively,
\begin{equation}
\label{eq: 1PK and Cauchy stress}
\ts{P}_{\s}^{e} = 2 \phi_{R_{\s}} \ts{F}_{\s} \frac{\partial \hat{\Psi}_{\s}}{\partial \ts{C}_{\s}}
\quad\text{and}\quad
\ts{T}_{\s}^{e} = \biggl(2 \phi_{R_{\s}} J_{\s}^{-1} \ts{F}_{\s} \frac{\partial \hat{\Psi}_{\s}}{\partial \ts{C}_{\s}} \transpose{\ts{F}}_{\s}\biggr)^{\circd{e}}.
\end{equation}
In addition to a hyperelastic contribution, we allow for a dissipative stress contribution due to the interaction of the solid with the fluid phase, analogous in form to what is found in Eq.~(2.1.3) in Ref.~\refcite{Bowen1976Theory-of-Mixtu0}. With this in mind, for $a = \s,\f$, we define
\begin{equation}
\label{eq: definition of Da}
\ts{D}_{a} \coloneqq \tfrac{1}{2} (\ts{L}_{a} + \transpose{\ts{L}}_{a}),\quad\text{with}\quad
\ts{L}_{a} \coloneqq \Grad (\bv{v}_{a}^{\circd{e}}).
\end{equation}
Then, the viscous contribution to the solid's stress in the mixture will be
\begin{equation}
\label{eq: solid viscous Cauchy stress}
\ts{T}_{\s}^{v} =  2 \mu_{B} (\ts{D}_{\s} - \ts{D}_{\f}).
\end{equation}
This type of contribution is common in poroelasticity, although it is usually discussed in relation to the fluid's behavior where it is referred to as a Brinkman term.\cite{Brinkman1949A-Calculation-o} The dynamic viscosity $\mu_{B}$ is generally a function of $\phi_{\s}$ with support identical to that of $\phi_{\s}$.

As far as the pure fluid phase is concerned, having already assumed that it is incompressible, we further assume that it is linear viscous, i.e., with viscous stress equal to $2 \hat{\mu}_{\f} \ts{D}_{\f}$, where $\hat{\mu}_{\f} \in \mathbb{R}^{+}_{0}$ is a constant dynamic shear viscosity. As for the stress response of the fluid in the mixture, we will assume that it consists of two terms: one with form $2 \mu_{\f} \ts{D}_{\f}$, where $\mu_{\f}$ is a dynamic viscosity that depends on $\phi_{\f}$, and the other being the counterpart of the Brinkman term introduced earlier (again, cf.\ Eq.~(2.1.3) in Ref.~\refcite{Bowen1976Theory-of-Mixtu0}):
\begin{equation}
\label{eq: viscous Cauchy stress of fluid in mixture}
\ts{T}_{\f}^{v} = 2 \mu_{\f} \ts{D}_{\f} + 2 \mu_{B} (\ts{D}_{\f} - \ts{D}_{\s}).
\end{equation}
The dynamic viscosity $\mu_{\f}$ will be subject to intuitive constraints such as $\mu_{\f}\to 0$ for $\phi_{\f}\to 0$ and $\mu_{\f}\to \hat{\mu}_{\f}$ for $\phi_{\f}\to 1$.

A pressure field (cf., e.g., Ref.~\refcite{Bowen1980Incompressible-0}) denoted by $p$ and called the \emph{pore pressure} is introduced over $B(t)$ to enforce Eq.~\eqref{eq: div vvavg constraint}. Then, the full forms of $\ts{T}_{\s}$ and $\ts{T}_{\f}$ are, respectively,
\begin{equation}
\label{eq: Ts and Tf overall}
\ts{T}_{\s} = -\phi_{\s} p \, \ts{I} + \ts{T}^{e}_{\s} +  \ts{T}^{v}_{\s}
\quad\text{and}\quad
\ts{T}_{\f} = -\phi_{\f} p \, \ts{I} + \ts{T}^{v}_{\f},
\end{equation}
where $\ts{I}$ is the identity tensor.

The constitutive equations for a mixture must also include expressions for $\bv{p}_{a}$ and $\bv{\xi}_{a}$ ($a = \s,\f$). In both cases, we will assume that these forces originate from a viscous resistance to the sliding of the fluid relative to the solid, i.e., they originate from the kinematic quantity $(\bv{v}_{\f} - \bv{v}_{\s})$.

For $\bv{p}_{\f}$ (by Eq.~\eqref{eq: interaction force constraint}, $\bv{p}_{\s} = -\bv{p}_{\f}$), we follow Ref.~\refcite{Bowen1980Incompressible-0} and write
\begin{equation}
\label{eq: pf and ff full form}
\bv{p}_{\f} = p \Grad{\phi_{\f}} + \bv{f}_{\f},
\quad\text{with}\quad
\bv{f}_{\f} = -\frac{\mu_{D} \phi_{\f}^{2}}{\kappa_{\s}} (\bv{v}_{\f}^{\circd{e}} - \bv{v}_{\s}^{\circd{e}}),
\end{equation}
where $\kappa_{\s}$ is the solid's permeability, while $\mu_{D}$ is a dynamic viscosity that, just like $\mu_{B}$, has support identical to that of $\phi_{\s}$. We defer the modeling of $\bv{\xi}_{a}$ to later in the paper.

\begin{remark}[Inherent vs.\ mixture constitutive behavior]
In Mixture Theory, one distinguishes the inherent behavior of a phase from how a phase behaves in the presence of another. This fact is reflected in Eqs.~\eqref{eq: solid viscous Cauchy stress},~\eqref{eq: viscous Cauchy stress of fluid in mixture} (second term on the right-hand side), and the second of Eqs.~\eqref{eq: pf and ff full form}, where we find well-known models for the fluid-solid \emph{interaction} proposed by Brinkman\cite{Brinkman1949A-Calculation-o} and Darcy\cite{Darcy1856Les-Fontaines-P}. When there is no interaction (see also  discussion in Section~5.1 of Ref.~\refcite{dellIsola2009Boundary-Condit-0}), the constituents' response reverts back to the one they have in their pure form. The consequences of this fact can be counter-intuitive. For example, in the traditional Darcy flow model one has an otherwise \emph{inviscid fluid} ($\hat{\mu}_{\f} = 0$ and $\mu_{B} = 0$), which takes on a viscous response  ($\mu_{D} > 0$) when flowing through a solid. The fact that the fluid model in the traditional Darcy model is inviscid is apparent in that the equations of motion do not provide a way to prescribe the boundary value of the tangential component of the fluid's velocity. This lack of control is also reflected in the classic choice of functional setting for the fluid's velocity, which is typically $H(\text{div})$ for the traditional Darcy problem (cf., e.g., Refs.~\refcite{Masud2002A-Stabilized-Mixed-0} and~\refcite{Hu2017A-Nonconforming}) rather than $H^{1}$ as one finds in the Stokes or Navier-Stokes problems.\cite{Wang2007New-Finite-Elem,Zhang2024BDM-Hrm-div-Wea,Ye2021A-Stabilizer-Fr,Masud2006A-Multiscale-Fi,Franca1992Stabilized-Finite-Element-1}
\end{remark}

\subsection{Additional consequences of momentum balance}
\label{Sec: TPE}
Here we derive an energy estimate following a process similar to that used in the \ac{TPE}.\cite{Gurtin-CMBook-1981-1,GurtinFried_2010_The-Mechanics_0}

We consider a time independent domain $\Omega \subset B(t)$ such that it is crossed by the singularity surface $\mathcal{S}(t)$. Letting $\mathcal{S}(t)_{\Omega} \coloneqq \Omega \cap \mathcal{S}(t)$, it is assumed that $\mathcal{S}(t)_{\Omega}$ is an open surface that separates $\Omega$ into two domains $\Omega^{+}(t)$ and $\Omega^{-}(t)$, such that the unit normal $\bv{m}$ is outward relative to $\Omega^{-}(t)$. Next, we take the inner product between Eq.~\eqref{eq: momentum} and $\bv{v}_{a}^{\circd{e}}$, and integrate over $\Omega$. Accounting for mass balance, and using integration by parts as well as the transport theorem, for $a = \s,\f$ we have
\begin{equation}
\label{eq: TPW inverse step 1}
\begin{multlined}[b]
\frac{\d{}}{\d{t}} \int_{\Omega}
k_{a}
+
\int_{\Omega}
(
\ts{T}_{a} \colondot \ts{L}_{a}
-
\bv{p}_{a} \cdot \bv{v}_{a}^{\circd{e}}
)
+
\int_{\mathcal{S}(t)}
\llbracket
\transpose{\ts{T}}_{a}\bv{v}_{a}
- 
k_{a}(\bv{v}_{a}^{\circd{e}} - \bv{v}_{\s}^{\circd{e}})
\rrbracket \cdot \bv{m}
\\
=
-\int_{\partial\Omega_{\text{ext}}}
(
k_{a} \bv{v}_{a}^{\circd{e}} \cdot \bv{n}
-
\ts{T}_{a}\bv{n} \cdot \bv{v}_{a}^{\circd{e}}
)
+
\int_{\Omega} \bv{b}_{a}\cdot\bv{v}_{a}^{\circd{e}},
\end{multlined}
\end{equation}
where $\bv{n}$ is the boundary outward unit normal and $k_{a}$ is the kinetic energy density of species $a$ per unit volume of the mixture, i.e.,
\begin{equation}
\label{eq: kinetic energy density def}
k_{a} \coloneqq \tfrac{1}{2} \rho_{a} \|\bv{v}_{a}^{\circd{e}}\|^{2}.
\end{equation}
To express our energy estimate in a compact manner, we first define the quantities
$\mathcal{W}_{\text{ext}}(\Omega)$, $\mathscr{D}$, and $\mathscr{I}$ as follows:
\begin{equation}
\label{eq: Externally supplied power}
\begin{aligned}[b]
\mathcal{W}_{\text{ext}}(\Omega) &\coloneqq
-\int_{\partial\Omega_{\text{ext}}}
\bigl[
(k_{\s} + \psi_{\s}) \bv{v}_{\s}^{\circd{e}}
+
k_{\f} \bv{v}_{\f}^{\circd{e}}
\bigr]\cdot \bv{n}
\\
&\qquad+
\int_{\partial\Omega_{\text{ext}}}
(
\ts{T}_{\s}\bv{n}\cdot\bv{v}_{\s}^{\circd{e}}
+
\ts{T}_{\f}\bv{n}\cdot\bv{v}_{\f}^{\circd{e}}
)
\\
&\qquad+
\int_{\Omega}
(
\bv{b}_{\s}\cdot\bv{v}_{\s}^{\circd{e}}
+
\bv{b}_{\f}\cdot\bv{v}_{\f}^{\circd{e}}
),
\end{aligned}
\end{equation}
\begin{equation}
\label{eq: Interior rate of dissipation}
\mathscr{D} \coloneqq 2 \mu_{\f} \|\ts{D}_{\f}\|^{2}
+
2 \mu_{B} \|\ts{D}_{\f}-\ts{D}_{\s}\|^{2}
+
\frac{\mu_{D} \phi_{\f}^{2}}{\kappa_{\s}} \|\bv{v}_{\f}^{\circd{e}} - \bv{v}_{\s}^{\circd{e}}\|^{2},
\end{equation}
and
\begin{equation}
\label{eq: power expended by the interface}
\mathscr{I} \coloneqq
\llbracket \ts{T}_{\f} \bv{m} \cdot (\bv{v}_{\f}^{\circd{e}} - \bv{v}_{\s}^{\circd{e}}) \rrbracket
-
\tfrac{1}{2} d \llbracket \|\bv{v}_{\f}^{\circd{e}} - \bv{v}_{\s}^{\circd{e}} \|^{2} \rrbracket.
\end{equation}
The quantities $\mathcal{W}_{\text{ext}}(\Omega)$, $\mathscr{D}$, and $\mathscr{I}$ denote, respectively, the external supply of mechanical power to $\Omega$, the volume density of the dissipation rate in $\Omega$, and the surface density of the power expenditure on the interface. Then, summing the contributions in Eq.~\eqref{eq: TPW inverse step 1} for the fluid and the solid phases, and using the constitutive response functions for the stress in both the solid and fluid phases, for the mixture, we have (see Proposition~\ref{claim: TPE} in the Appendix)
\begin{equation}
\label{eq: TPE Mixture}
\mathcal{W}_{\text{ext}}(\Omega)
=
\frac{\d{}}{\d{t}} \int_{\Omega} (k_{\s} + k_{\f} + \psi_{\s}) + \int_{\Omega} \mathscr{D} + \int_{\mathcal{S}(t)} \mathscr{I}.
\end{equation}
\begin{remark}[Theorem of power expended]
    Equation~\eqref{eq: TPE Mixture} is the version of the theorem of power expended (cf.\ Ref.~\refcite{Gurtin-CMBook-1981-1}) for the present context. It presents a consequence of the momentum balance laws as opposed to, say, the balance of energy. In a way, the result in Eq.~\eqref{eq: TPE Mixture} is intuitive: the right-hand side tells us what happens to the  power input to $\Omega$, which is the left-hand side. It also allows to consider the power expended at the interface in a transparent manner, as argued in the following remarks.
\end{remark}

\begin{remark}[Inert vs.\ purely dissipative interfaces]
\label{remark: Inert vs dissipative interface}
     We call \emph{inert} an interface which neither reversibly stores nor does it dissipate power. In this case, appealing to the arbitrariness of $\Omega$, the term $\mathscr{I}$ in Eq.~\eqref{eq: TPE Mixture} must vanish:
\begin{equation}
\label{eq: jump condition of inert interfaces}
\llbracket \ts{T}_{\f} \bv{m} \cdot (\bv{v}_{\f}^{\circd{e}} - \bv{v}_{\s}^{\circd{e}}) \rrbracket 
-
\tfrac{1}{2} d \llbracket \|\bv{v}_{\f}^{\circd{e}} - \bv{v}_{\s}^{\circd{e}} \|^{2} \rrbracket
= 0.
\end{equation}
Provided that the constitutive assumptions in the present work are different from those in Ref.~\refcite{dellIsola2009Boundary-Condit-0}, \emph{mutatis mutandis}, Eq.~\eqref{eq: jump condition of inert interfaces} coincides with Eq.~(33) in Ref.~\refcite{dellIsola2009Boundary-Condit-0} when the interface is inert. By contrast, if the interface were purely dissipative, then $\mathscr{I}$ would 
measure the dissipation at the interface:
\begin{equation}
\label{eq: dissipative interface case}
\mathscr{I} = \mathscr{D}_{\mathcal{S}(t)},
\end{equation}
where $\mathscr{D}_{\mathcal{S}(t)}$ is the dissipation rate per unit area of the interface. This conclusion again agrees with Eq.~(33) in Ref.~\refcite{dellIsola2009Boundary-Condit-0}, which does include dissipation.
\end{remark}

\begin{remark}[Development of a constitutive theory for the interface forces]
\label{remark: interface constitutive equations}
    The result in Eq.~\eqref{eq: TPE Mixture} would be an important element in developing a constitutive theory of the interface. In this particular treatment, having assumed that the motion is $C^{1}_{\text{pw}}$ and treating $\bv{v}_{\s}$ as globally continuous, the development of a constitutive relation for $\bv{\xi}_{a}$ would start with positing that $\bv{\xi}_{a}$ can be a function of the kinematic variables that experience a discontinuity at the interface, e.g., something like $\bv{\xi}_{a} = \bv{\xi}_{a}(\llbracket \phi_{\f} \rrbracket, \llbracket\bv{v}_{\f}\rrbracket, \llbracket \ts{D}_{\s} \rrbracket, \llbracket \ts{D}_{\f} \rrbracket, \ldots)$. Another possible strategy is that presented in Ref.~\refcite{dellIsola2009Boundary-Condit-0}, where a constitutive assumption is made about the dissipation. For the purpose of obtaining a consistent formulation that can be more readily translated into a \ac{FEM}, we will follow the latter strategy.
\end{remark}

\begin{remark}[Concerning the continuity of the pore pressure across $\mathcal{S}(t)$]
    As articulated in Ref.~\refcite{Shim2022A-Hybrid-Biphas}, in some cases, it is quite reasonable to expect the pore pressure $p$ to be continuous across $\mathcal{S}(t)$. However, there are also cases in which a proper treatment of osmosis requires the possibility that the pore pressure be discontinuous across $\mathcal{S}(t)$. It is our position that if the constitutive behavior of the interface implies $\llbracket p \rrbracket = 0$, the enforcement of this condition should arise naturally from the overall model formulation. In the next Section, we present a formulation based on Ref.~\refcite{dellIsola2009Boundary-Condit-0} with this feature. To exemplify this point, suppose that we have an inert interface so that Eq.~\eqref{eq: jump condition of inert interfaces} holds. Suppose further that the pure fluid is inviscid as for the familiar Darcy flow, i.e., $\ts{T}_{\f} = - p \phi_{\f} \ts{I}$. Finally, suppose that the inertia of the fluid is negligible. In this case, Eq.~\eqref{eq: jump condition of inert interfaces} reduces to
\begin{equation*}
 \llbracket p \phi_{\f} \bv{m} \cdot (\bv{v}_{\f}^{\circd{e}} - \bv{v}_{\s}^{\circd{e}}) \rrbracket = 0
 \quad\Rightarrow\quad
 \llbracket p \rrbracket (\bv{v}_{\flt} \cdot \bv{m}) = 0,
\end{equation*}
where we have used the fact that the normal component of the filtration velocity is continuous as a consequence of the balance of mass. If $(\bv{v}_{\flt} \cdot \bv{m}) \neq 0$, we would then conclude that $\llbracket p \rrbracket = 0$. This result is a ``classic'' result (cf., e.g., Ref.~\refcite{Hou1989Boundary-condit}) that is recovered in the current treatment without enforcing it strongly by representing $p$ as a globally continuous field.
\end{remark}

\section{A Variational Formulation}
\label{sec: variational formulations}
Here we derive the field equations using Hamilton's principle (cf.\ Ref.~\refcite{Bedford2021Hamiltons-Princ}) following the approach in Ref.~\refcite{dellIsola2009Boundary-Condit-0}. The use of Hamilton's principle in poroelasticity is not new (cf.\ Ref.~\refcite{Bedford2021Hamiltons-Princ}). Our purpose in somewhat duplicating the derivation of the field equations is (\emph{i}) to provide a connection between the Newtonian approach (cf.\ Ref.~\refcite{Gurtin2000Configurational0}) to derive the field equations and a variational one; and (\emph{ii}) to find a simple path to their numerical implementation using \ac{FEM}. As mentioned in Remark~\ref{remark: interface constitutive equations}, we wish to translate assumptions on the interface constitutive response to energy estimates for our \ac{FEM} formulation in a direct way.

\subsection{Hamilton-Rayleigh Principle}
\label{subsec: Hamilton principle}
In the chosen kinematic framework, the system's configuration can be equivalently described by the pair $\{\bv{\chi}_{\s},\bv{\chi}_{\f}\}$ (cf.\ Eq.~\eqref{eq: Motions defs}) or the pair $\{\bv{\chi}_{\s},\bv{\chi}_{\frs}\}$ (cf.\ Eq.~\eqref{eq: Relative motion}). Recalling that the inverse image of the interface $\mathcal{S}(t)$ under $\bv{\chi}_{\s}(\bv{X}_{\s},t)$ is time independent, and observing that each element of $\{\bv{\chi}_{\s},\bv{\chi}_{\frs}\}$ has the same spatial domain $B_{\s}$ and is continuous over it, we choose 
\begin{equation}
\label{eq: hamiltons motion pair}
q(t) \coloneqq \{\bv{\chi}_{\s},\bv{\chi}_{\frs}\}
\end{equation}
as the (abstract) motion of our system for the purpose of applying Hamilton's principle.

Letting $\dot{q} \coloneqq \partial_{t} q$, and interpreting differentiation with respect to $q$ and $\dot{q}$ in a variational sense, for any sufficiently smooth $q$ in time and $\delta q$ with compact support in the time interval $[0,T]$, we require
\begin{equation}
\label{eq: Hamiltons principle paradigm}
\int_{0}^{T}
\frac{\partial (\mathscr{K} - \mathscr{U})}{\partial q} \delta q + \int_{0}^{T}\frac{\partial (\mathscr{C} + \mathscr{C}_{\flt} - \mathscr{R})}{\partial \dot{q}} \delta q = 0,
\end{equation}
where, assuming for simplicity (and without loss of generality) that the boundary conditions consist of homogeneous Dirichlet data for both components of the motion, $\mathscr{K}$, $\mathscr{U}$, $\mathscr{C}$, $\mathscr{C}_{\flt}$, and $\mathscr{R}$ are the following functionals:
\begin{align}
\label{eq: K definition}
\mathscr{K} &\coloneqq \int_{B_{\s}} (k_{\s} + k_{\f})^{\circd{s}} J_{\s},
\\
\label{eq: U definition}
\mathscr{U} &\coloneqq \int_{B_{\s}} (\Psi_{\s} - \bv{u}_{\s} \cdot \bv{b}_{\s}^{\circd{s}} J_{\s} - \bv{u}_{\f}^{\circd{s}} \cdot \bv{b}_{\f}^{\circd{s}} J_{\s}),
\\
\label{eq: Constraint definition}
\mathscr{C} &\coloneqq \int_{B_{\s}} p^{\circd{s}} \, J_{\s} [\Div{(\bv{v}_{\s}^{\circd{e}} + \bv{v}_{\flt})}]^{\circd{s}},
\\
\label{eq: Surface constraint definition}
\mathscr{C}_{\flt} &\coloneqq \int_{S_{\s}} \ip^{\circd{s}} \, \llbracket\bv{v}_{\flt}^{\circd{s}} \cdot J_{\s}\inversetranspose{\ts{F}}_{\s} \bv{m}_{\s}\rrbracket,
\\
\label{eq: R potential}
2 \mathscr{R} &\coloneqq 
\int_{B_{\s}} \mathscr{D}^{\circd{\s}} J_{\s}
+
\int_{\mathcal{S}_{\s}} (\mathscr{D}_{\mathcal{S}(t)})^{\circd{s}} J_{\s} \|\inversetranspose{\ts{F}}_{\s} \bv{m}_{\s}\|.
\end{align}
$\mathscr{K}$ and $\mathscr{U}$ are the system's kinetic and potential energy, respectively. 
$\mathscr{C}$ encodes constituent's (individual) incompressibility, whereas letting
\begin{equation}
\label{eq: ip def}
\ip: \mathcal{S}(t) \to \mathbb{R}
\end{equation}
be a Lagrange multiplier, $\mathscr{C}_{\flt}$ encodes the requirement that the normal component of the filtration velocity be continuous across $\mathcal{S}(t)$. $\mathscr{R}$ is the time rate of energy dissipation.  The term $\mathscr{D}_{\mathcal{S}(t)}$ represents the interface dissipation mentioned in Remark~\ref{remark: Inert vs dissipative interface}, for which we adopt the following (simple) quadratic form:
\begin{equation}
\label{eq: Surface Dissipation potential}
\mathscr{D}_{\mathcal{S}(t)} \coloneqq \tfrac{1}{2} 
\mu_{\mathcal{S}}
\|\llbracket
\bv{v}_{\flt}\rrbracket\|^{2},
\end{equation}
with $\mu_{\mathcal{S}}$ a positive constant with dimensions of a dynamic viscosity per unit length.
\begin{remark}[On the constitutive response of the interface]
The choice in Equation~\eqref{eq: Surface Dissipation potential} is intended as an extremely simple example of how the constitutive response of the interface is built into the formulation. In this case we are positing that the interface is purely dissipative. As always, different choices for $\mathscr{D}_{\mathcal{S}(t)}$ are possible, depending on the available kinematic framework. In fact, by properly modifying $\mathscr{U}$ in Eq.~\eqref{eq: U definition}, one can also describe interfaces capable of storing energy reversibly. Going back to Eq.~\eqref{eq: Surface Dissipation potential},
recalling that the continuity of the normal component of the filtration velocity is enforced through the jump condition of the balance of mass, the chosen expression of $\mathscr{D}_{\mathcal{S}(t)}$ can be interpreted as a penalization of $\llbracket \bv{v}_{\flt}\rrbracket$ in the tangential direction. In turn, this provides a strategy for the weak enforcement of the full continuity of the filtration velocity, a condition discussed, e.g., in Refs.~\refcite{Hou1989Boundary-condit} and~\refcite{Shim2022A-Hybrid-Biphas}. This said, we note that we cannot expect that the chosen form of energy dissipation will yield a direct control of the jump discontinuity of the pore pressure, simply because we did not include  terms designed to achieve such control. To some extent, our chosen model is also intended to illustrate this point. That is, if the jump discontinuity of the pore pressure is to be controlled, then a specific constitutive mechanism responsible for such control must be posited. The last example in Sec.~\ref{Sec: Examples} will further illustrate this point by showing that increasing the value of $\mu_{\mathcal{S}}$ does yield a decrease of the norm of the jump of the tangential component of the filtration velocity but it does not have a similar control on the norm of the jump of the pore pressure. 
\end{remark}

Applying standard methods, under the stated assumptions (cf.\ also Ref.~\refcite{dellIsola2009Boundary-Condit-0}), Eqs.~\eqref{eq: Hamiltons principle paradigm}--\eqref{eq: Surface Dissipation potential} yield the following result (cf.\ Proposition~\ref{claim: hamiltons principle} in the Appendix):
\begin{equation}
\label{eq: Hamiltons principle final global}
\begin{aligned}[b]
0 &= \int_{B_{\s}}
J_{\s} \bigl[
    \rho_{\s} \bv{a}_{\s}^{\circd{e}} + \rho_{\f} \bv{a}_{\f}^{\circd{e}} - \bv{b}_{\s} - \bv{b}_{\f} - \Div{(\ts{T}_{\s} + \ts{T}_{\f})}\bigr]^{\circd{s}} \cdot \delta\bv{\chi}_{\s}
\\
&\quad-
\int_{B_{\s}} J_{\s} (\rho_{\f} \bv{a}_{\f}^{\circd{e}} - \bv{b}_{\f} - \Div{\ts{T}_{\f}} - \bv{p}_{\f})^{\circd{s}} \cdot \ts{F}_{\s}\ts{F}_{\frs}^{-1} \delta\bv{\chi}_{\frs}
\\
&\quad+
\int_{\mathcal{S}_{\s}}
\llbracket
\rho_{\f}(\bv{v}_{\f}^{\circd{e}} - \bv{v}_{\s}^{\circd{e}}) \otimes (\bv{v}_{\f}^{\circd{e}} - \bv{v}_{\s}^{\circd{e}}) - \ts{T}_{\s} - \ts{T}_{\f} 
\rrbracket^{\circd{s}} J_{\s} \inversetranspose{\ts{F}}_{\s}  \bv{m}_{\s}
\cdot \delta\bv{\chi}_{\s}
\\
&\quad+
\int_{\mathcal{S}_{\s}}
\biggl\llbracket
\Bigl(
k_{\f} \ts{I}
-
\rho_{\f} \bv{v}_{\f}^{\circd{e}} \otimes (\bv{v}_{\f}^{\circd{e}} - \bv{v}_{\s}^{\circd{e}})
\Bigr)^{\circd{s}} J_{\s} \inversetranspose{\ts{F}}_{\s} \bv{m}_{\s} \cdot \ts{F}_{\s}\ts{F}_{\frs}^{-1} \delta \bv{\chi}_{\frs}
\biggr\rrbracket
\\
&\quad+
\int_{\mathcal{S}_{\s}}
\Bigl\llbracket
\bigl(
\phi_{\f} \ip \, \ts{I} + 
\ts{T}_{\f}
-
\tfrac{1}{2} \phi_{\f} \mu_{\mathcal{S}} \llbracket \bv{v}_{\flt} \rrbracket 
\otimes \bv{m}
\bigr)^{\circd{s}} J_{\s}\inversetranspose{\ts{F}}_{\s}\bv{m}_{\s}
\cdot \ts{F}_{\s}\ts{F}_{\frs}^{-1} \delta \bv{\chi}_{\frs}
\Bigr\rrbracket.
\end{aligned}
\end{equation}
By taking admissible variations of the fields $\delta\bv{\chi}_{\s}$ and $\delta\bv{\chi}_{\frs}$, we then conclude that
\begin{align}
\label{eq: balance of momentum mixture from variation}
\rho_{\s} \bv{a}_{\s}^{\circd{e}} + \rho_{\f} \bv{a}_{\f}^{\circd{e}} - \bv{b}_{\s} - \bv{b}_{\f} - \Div{(\ts{T}_{\s} + \ts{T}_{\f})} &= \bv{0}\quad\text{in $B(t)$},
\\
\label{eq: balance of momentum fluid from variation}
\rho_{\f} \bv{a}_{\f}^{\circd{e}} - \bv{b}_{\f} - \Div{\ts{T}_{\f}} - \bv{p}_{\f} &= \bv{0}\quad\text{in $B(t)$},
\\
\label{eq: balance of momentum jump mixture from variation}
\llbracket
\rho_{\f}(\bv{v}_{\f}^{\circd{e}} - \bv{v}_{\s}^{\circd{e}}) \otimes (\bv{v}_{\f}^{\circd{e}} - \bv{v}_{\s}^{\circd{e}}) - \ts{T}_{\s} - \ts{T}_{\f}
\rrbracket \bv{m}
&= \bv{0}\quad\text{on $\mathcal{S}(t)$},
\shortintertext{and}
\label{eq: balance of momentum jump fluid from variation}
\begin{multlined}[b]
\Bigl\llbracket
\inversetranspose{\ts{F}}_{\frs}
\transpose{\ts{F}}_{\s}
\Bigl(
k_{\f} \ts{I}
-
\rho_{\f} \bv{v}_{\f}^{\circd{e}} \otimes (\bv{v}_{\f}^{\circd{e}} - \bv{v}_{\s}^{\circd{e}})
+
\phi_{\f} \ip \, \ts{I} + 
\ts{T}_{\f}
\\
-
\tfrac{1}{2} \phi_{\f} \mu_{\mathcal{S}} \llbracket \bv{v}_{\flt} \rrbracket \otimes \bv{m}
\Bigr)^{\circd{s}} \,
\Big\rrbracket \bv{m}^{\circd{s}}
\end{multlined}
&= \bv{0} 
\quad
\text{on $\mathcal{S}_{\s}$}.
\end{align}

\begin{remark}[Concerning Eqs.~\eqref{eq: balance of momentum mixture from variation} and~\eqref{eq: balance of momentum fluid from variation}]
Referring to Eqs.~\eqref{eq: momentum} and~\eqref{eq: interaction force constraint}, one can see that Eq.~\eqref{eq: balance of momentum mixture from variation} expresses the balance of momentum for the mixture, the latter obtained by summing Eq.~\eqref{eq: momentum} for the two constituents and enforcing the constraint in Eq.~\eqref{eq: interaction force constraint}; whereas Eq.~\eqref{eq: balance of momentum fluid from variation} is directly the expression of the balance of momentum for the fluid phase. These equations correspond to Eqs.~(28) and~(29) in Ref.~\refcite{dellIsola2009Boundary-Condit-0}.
\end{remark}

\begin{remark}[Concerning Eq.~\eqref{eq: balance of momentum jump mixture from variation}]
Equation~\eqref{eq: balance of momentum jump mixture from variation} expresses the jump condition of the balance of momentum for the mixture as a whole. This equation coincides with Eq.~\eqref{eq: Jump condition for mixture}. \end{remark}

\begin{remark}[Concerning Eq.~\eqref{eq: balance of momentum jump fluid from variation}]
\label{remark: energy estimate from Hamilton}
Equation~\eqref{eq: balance of momentum jump fluid from variation} corresponds to Eq.~(31) in Ref.~\refcite{dellIsola2009Boundary-Condit-0} and it can be treated following the arguments presented therein to break up the equation into components. The object on the left-hand side of Eq.~\eqref{eq: balance of momentum jump fluid from variation} takes values in the cotangent space of points on $\mathcal{S}_{\f}(t)$. One component of this quantity can be obtained by taking a convenient tangent vector to $\mathcal{S}_{\f}(t)$ acting on it. To select such a vector, we recall that the continuity of $\bv{\chi}_{\s}$ implies:
\begin{equation}
\label{eq: regularity of chis}
\llbracket \ts{F}_{\s} \rrbracket \bv{\tau}_{\s_{1}} = \bv{0},
\quad
\llbracket \ts{F}_{\s} \rrbracket \bv{\tau}_{\s_{2}} = \bv{0},
\quad
\llbracket J_{\s} \inversetranspose{\ts{F}}_{\s} \bv{m}_{\s}\rrbracket = \bv{0},
\quad\text{and}\quad
\llbracket \bv{v}_{\s} \rrbracket = \bv{0}.
\end{equation}
Furthermore, under the kinematic framework in Ref.~\refcite{dellIsola2009Boundary-Condit-0}, the following conditions also hold:
\begin{equation}
\label{eq: continuity of the orienting set}
\llbracket\ts{F}_{\f}^{-1}\ts{F}_{\s}\bv{\tau}_{\s_{1}}\rrbracket = \bv{0},\quad
\llbracket\ts{F}_{\f}^{-1}\ts{F}_{\s}\bv{\tau}_{\s_{2}}\rrbracket
= \bv{0},\quad\text{and}\quad
\llbracket\inversetranspose{J_{\f}^{-1}J_{\s}(\ts{F}_{\f}^{-1}\ts{F}_{\s})}\bv{m}_{\s}\rrbracket = \bv{0}.
\end{equation}
Then, letting $\bv{\tau}$ be a tangent vector at $\hat{\bv{x}} \in \mathcal{S}(t)$, a suitable tangent vector to $\mathcal{S}_{\f}(t)$ would be $\ts{F}_{\frs}\ts{F}_{\s}^{-1}\bv{\tau}$ and its action on the left-hand side of Eq.~\eqref{eq: balance of momentum jump fluid from variation}, once remapped on $\mathcal{S}(t)$, yields
\begin{equation*}
\bv{\tau}\cdot
\Bigl(
d\llbracket
\bv{v}_{\f}^{\circd{e}}
\rrbracket
-
\llbracket
\ts{T}_{\f}
\rrbracket
\bv{m}
+
\tfrac{1}{2} \llbracket\phi_{\f}\rrbracket \mu_{\mathcal{S}} \llbracket \bv{v}_{\flt}\rrbracket
\Bigr)
=
0
\quad\text{on $\mathcal{S}(t)$}.
\end{equation*}
The above equation coincides with the projection of the second of Eqs.~\eqref{eq: Jump condition for solid and fluids individually} along $\bv{\tau}$. Letting $\mathbb{P}_{\tau} = \ts{I} - \bv{m} \otimes \bv{m}$ denote the projection tangent to $\mathcal{S}(t)$, the above equation can then be given in the more general form
\begin{equation}
\label{eq: tangent component of jump condition momentum of fluid}
\mathbb{P}_{\tau} 
\Bigl(
d\llbracket
\bv{v}_{\f}^{\circd{e}}
\rrbracket
-
\llbracket
\ts{T}_{\f}
\rrbracket
\bv{m}
+
\tfrac{1}{2} \llbracket\phi_{\f}\rrbracket\mu_{\mathcal{S}} \llbracket \bv{v}_{\flt}\rrbracket
\Bigr)
=
\bv{0}
\quad\text{on $\mathcal{S}(t)$}.
\end{equation}
Equation~\eqref{eq: tangent component of jump condition momentum of fluid} is useful in providing clarity as to the nature of the surface force $\bv{\xi}_{\f}$ introduced earlier. Clearly $\tfrac{1}{2}\llbracket\phi_{\f}\rrbracket \mu_{\mathcal{S}} \mathbb{P}_{\tau} \llbracket\bv{v}_{\flt}\rrbracket$ is only a component of $\bv{\xi}_{\f}$. As pointed out in Ref.~\refcite{dellIsola2009Boundary-Condit-0}, if $d \neq 0$, another (independent) component can be derived by considering the action of the vector $\ts{F}_{\frs}\ts{F}_{\s}^{-1}(\bv{v}_{\f}^{\circd{e}}-\bv{v}_{\s}^{\circd{e}})$ (cf.\ the last of Eqs.~\eqref{eq: vf rel vs}) onto the left-hand side of Eq.~\eqref{eq: balance of momentum jump fluid from variation}. The vector in question is continuous across $\mathcal{S}_{\s}$ and not tangent to $\mathcal{S}_{\f}$ when $d \neq 0$.
\end{remark}

\begin{remark}[On the force $\bv{\xi}_{a}$]
\label{remark: on the force xa}
    The force $\bv{\xi}_{a}$ was introduced in Eq.~\eqref{eq: jump of momentum}. We now present an argument for the determination of an explicit form for $\bv{\xi}_{a}$. We begin by observing that the result in Eq.~\eqref{eq: balance of momentum fluid from variation} is derived from the second line of Eq.~\eqref{eq: Hamiltons principle final global} by choosing an arbitrary variation $\delta\bv{\chi}_{\frs}$ and relying on the fact that $\ts{F}_{\s}$ and $\ts{F}_{\frs}$ are invertible. By contrast, Eq.~\eqref{eq: balance of momentum jump fluid from variation} does not allow us to proceed in the same manner, but requires careful scrutiny on the freedom with which one can take variations of the motion of the fluid relative to the solid. This said, choosing the motion of the fluid relative to the solid to be discontinuous  then the test function
\begin{equation*}
\ts{F}_{\s}\ts{F}_{\frs}^{-1} \delta\bv{\chi}_{\frs}
\end{equation*}
can be given arbitrary discontinuous values across the interface. Clearly, taken in isolation, this statement can imply some unphysical situations. However, the overall effect of this assumption is tempered by having imposed the constraint in Eq.~\eqref{eq: div vvavg constraint expanded} explicitly through the term in Eq.~\eqref{eq: Surface constraint definition}. It is our view that this way of proceeding is similar to the imposition of the incompressibility constraint in, say, the mixed weak formulation for the Stokes problem (cf.\ Ref.~\refcite{Brezzi1991Mixed-and-Hybrid-0}) where the velocity field and associated test functions are taken in spaces that are not necessarily divergence-free. This greater freedom in the choice of velocity function spaces is corrected by introducing a Lagrange multiplier that has the responsibility to enforce the divergence-free constraint. In our application, the Lagrange multiplier entrusted with ensuring conformance to the jump condition of the balance of mass is $\ip$.

With our choice of test functions, then the following result is obtained:
\begin{equation}
\label{eq: final abstract form of the jump condition of the fluid momentum}
\Bigl(
k_{\f}\bv{m}
- d \bv{v}_{\f}^{\circd{e}}
+
\phi_{\f} \ip \bv{m} + 
\ts{T}_{\f}\bv{m}
-
\tfrac{1}{2} \phi_{\f} \mu_{\mathcal{S}}
\llbracket \bv{v}_{\flt}\rrbracket
\Bigr)^{\pm}
= \bv{0} 
\quad
\text{on $\mathcal{S}(t)$}.
\end{equation}
We note that, under the assumption that $d = 0$, Eqs.~\eqref{eq: final abstract form of the jump condition of the fluid momentum} have similarities with Eqs.~(35) in Ref.~\refcite{dellIsola2009Boundary-Condit-0}. Next, referring to the second of Eqs.~\eqref{eq: Jump condition for solid and fluids individually}, using Eqs.~\eqref{eq: final abstract form of the jump condition of the fluid momentum}, we would conclude that
\begin{equation}
\label{eq: xif constitutive relation}
\bv{\xi}_{\f} = \tfrac{1}{2} \mu_{\mathcal{S}}
\llbracket \phi_{\f} \rrbracket
\llbracket \bv{v}_{\flt} \rrbracket
-
\llbracket k_{\f} + \phi_{\f} \ip \rrbracket \bv{m}.
\end{equation}
The above form of $\bv{\xi}_{\f}$ is such that the first term on the right-hand side of Eq.~\eqref{eq: xif constitutive relation} is tangent to the interface whereas the second term is normal to it.
\end{remark}

\subsection{A practical weak formulation}
Here we propose a weak formulation of the problem based on Eq.~\eqref{eq: Hamiltons principle final global} and the kinematic framework adopted in Remark~\ref{remark: on the force xa}. For the purpose of this paper, we will use the same simple constitutive assumptions made in the discussion of Hamilton's principle. Clearly, these assumptions can be changed to fit the modeling needs of a particular problem.

With the above in mind, the principal unknowns of our problem are selected as follows:
\begin{equation}
\label{eq: principal unknowns}
\bv{u}_{\s},\quad
\bv{v}_{\f}^{\circd{s}},\quad
p^{\circd{s}},\quad\text{and}\quad
\ip^{\circd{s}}.
\end{equation}
This choice of principal unknowns is meant to produce a formulation that is close to familiar \ac{FEM} formulations of poroelasticity. That is, we do not perceive that a \ac{FEM} in which the primary unknowns are $\bv{\chi}_{\s}$ and $\bv{\chi}_{\frs}$ would be adaptable to most applications. As for the particular form of the list,
Eq.~\eqref{eq: principal unknowns} is meant to indicate that we are proposing an \ac{ALE} formulation where the field variables are understood to be defined over $B_{\s}$, the latter domain coinciding with our computational domain. At the same time, to avoid unnecessary repetition, we will keep using the same symbols as in the formulation section of the paper. Also, to avoid proliferation of symbols, we denote the test functions of said fields by a tilde.

While the regularity of these fields has been discussed earlier in the paper, we recognize that the formal properties of these fields must be examined on a case by case basis depending on the specific set of constitutive equations that are selected for a corresponding specific problem. In fact, the constitutive equations also have repercussions to the type of boundary conditions that one might impose, as it is clear from the comparison between, say, the traditional Stokes and Darcy problems. With this in mind, and referring to Eq.~\eqref{eq: Hamiltons principle final global}, we operate the following substitution of test functions:
\begin{equation}
\label{eq: test function substitution}
\delta\bv{\chi}_{\s} \to \tilde{\bv{u}}_{\s}/\theta
\quad\text{and}\quad
-\ts{F}_{\s}\ts{F}_{\frs}^{-1} \delta\bv{\chi}_{\frs} \to \tilde{\bv{v}}_{\f},
\end{equation}
where $\theta$ is a unit time constant whose only purpose is to achieve dimensional consistency. We then propose the following weak formulation of the problem, which incorporates a weak enforcement of the constraints in the second of Eqs.~\eqref{eq: div vvavg constraint expanded} and the first of Eqs.~\eqref{eq: mass balance jump condition}. For simplicity of presentation, we assume that Dirichlet boundary conditions are prescribed for $\bv{u}_{\s}$ and $\bv{v}_{\f}$ on the external boundary of the mixture. Such weak formulation is then given by
\begin{equation}
\label{eq: Abstract FEM mixture}
\begin{aligned}[b]
0 &= \int_{B_{\s}}
J_{\s} \bigl(
    \rho_{\s} \bv{a}_{\s}^{\circd{e}} + \rho_{\f} \bv{a}_{\f}^{\circd{e}} - \bv{b}_{\s} - \bv{b}_{\f} + \Grad{p}
    \bigr)^{\circd{s}} 
    \cdot \tilde{\bv{u}}_{\s}/\theta
    \\
    &\quad
    + \int_{B_{\s}}
    J_{\s}
    (\ts{T}_{\s}^{e} + \ts{T}_{\s}^{v} + \ts{T}_{\f}^{v})^{\circd{s}} \inversetranspose{\ts{F}}_{\s} \colondot \Grad{\tilde{\bv{u}}_{\s}}/\theta
\\
&\quad+
\int_{B_{\s}}
J_{\s} \bigl[
 (\rho_{\f} \bv{a}_{\f}^{\circd{e}} - \bv{b}_{\f} + \phi_{\f}\Grad{p} - \bv{f}_{\f})^{\circd{s}} \cdot \tilde{\bv{v}}_{\f}^{\circd{s}}
+ (\ts{T}_{\f}^{v})^{\circd{s}} \inversetranspose{\ts{F}}_{\s} \colondot \Grad{\tilde{\bv{v}}_{\f}^{\circd{s}}}\bigr]
\\
&\quad+
\int_{\mathcal{S}_{\s}}
\llbracket
\rho_{\f}(\bv{v}_{\f}^{\circd{e}} - \bv{v}_{\s}^{\circd{e}}) \otimes (\bv{v}_{\f}^{\circd{e}} - \bv{v}_{\s}^{\circd{e}}) + p \ts{I}
\rrbracket^{\circd{s}} J_{\s} \inversetranspose{\ts{F}}_{\s}  \bv{m}_{\s}
\cdot \tilde{\bv{u}}_{\s}/\theta
\\
&\quad-
\int_{\mathcal{S}_{\s}}
\biggl\llbracket
\Bigl(
k_{\f} \ts{I}
-
\rho_{\f} \bv{v}_{\f}^{\circd{e}} \otimes (\bv{v}_{\f}^{\circd{e}} - \bv{v}_{\s}^{\circd{e}})
\Bigr)^{\circd{s}} J_{\s} \inversetranspose{\ts{F}}_{\s} \bv{m}_{\s} \cdot \tilde{\bv{v}}_{\f}^{\circd{s}}
\biggr\rrbracket
\\
&\quad-
\int_{\mathcal{S}_{\s}}
\Bigl\llbracket
\bigl(
\phi_{\f} (\ip - p) \, \ts{I}
-
\tfrac{1}{2} \phi_{\f} \mu_{\mathcal{S}} \llbracket \bv{v}_{\flt} \rrbracket 
\otimes \bv{m}
\bigr)^{\circd{s}} J_{\s}\inversetranspose{\ts{F}}_{\s}\bv{m}_{\s}
\cdot \tilde{\bv{v}}_{\f}^{\circd{s}}
\Bigr\rrbracket
\\
&\quad
-
\int_{B_{\s}}
J_{\s} (\bv{v}_{\s} + \bv{v}_{\flt}^{\circd{s}}) \cdot \inversetranspose{\ts{F}}_{\s}\Grad{\tilde{p}^{\,\circd{s}}}
-
\int_{\mathcal{S}_{\s}} 
\llbracket
(\bv{v}_{\s} +
\bv{v}_{\flt}^{\circd{s}}) \, \tilde{p}^{\circd{s}}
\rrbracket \cdot J_{\s} \inversetranspose{\ts{F}}_{\s} \bv{m}_{\s}
\\
&\quad
+
\int_{\mathcal{S}_{\s}} 
\llbracket
\bv{v}_{\flt}^{\circd{s}}
\rrbracket
\tilde{\ip}^{\circd{s}} \cdot
J_{\s} \inversetranspose{\ts{F}}_{\s} \bv{m}_{\s}
+
\int_{\partial B_{\s_{\text{ext}}}} \tilde{p}^{\circd{s}} (\bv{v}_{\s} + \bv{v}_{\flt}^{\circd{s}}) \cdot J_{\s} \inversetranspose{\ts{F}}_{\s} \bv{n}_{\s},
\end{aligned}
\end{equation}
where $\partial B_{\s_{\text{ext}}}$ denotes the external boundary of $B_{\s}$.

\begin{proposition}[Formal stability estimate]
\label{prop: formal energy estimate}
    Setting to zero the body forces and the boundary data, if, for almost every $t\in [0,T]$, $\bv{u}_{\s}$, $\bv{v}_{\f}^{\circd{s}}$, $p^{\circd{s}}$, and $\ip^{\circd{s}}$ solve Eq.~\eqref{eq: Hamiltons principle final global}, then the following energy estimate holds:
\begin{equation}
\label{eq: formal stability estimate}
\frac{\d{}}{\d{t}} \int_{B(t)} (k_{\s} + k_{\f} + \psi_{\s}) + \int_{B(t)} \mathscr{D} + \int_{\mathcal{S}(t)} \mathscr{D}_{\mathcal{S}(t)} = 0.
\end{equation}
\end{proposition}

\begin{proof}
    Set
\begin{equation*}
\tilde{\bv{u}}_{\s}/\theta = \bv{v}_{\s},\quad
\tilde{\bv{v}}_{\f}^{\circd{s}} = \bv{v}_{\f}^{\circd{s}} - \bv{v}_{\s},\quad
\tilde{p}^{\circd{s}} = p^{\circd{s}},
\quad\text{and}\quad
\tilde{\ip}^{\circd{s}} = \ip^{\circd{s}}.
\end{equation*}
This substitution yields
\begin{equation}
\label{eq: Proof of energy estimate step 1}
\begin{aligned}[b]
0 &= \int_{B_{\s}}
J_{\s} \bigl(
    \rho_{\s} \bv{a}_{\s}^{\circd{e}} + \rho_{\f} \bv{a}_{\f}^{\circd{e}}
    \bigr)^{\circd{s}} 
    \cdot \bv{v}_{\s}
    \\
    &\qquad
    + \int_{B_{\s}}
    J_{\s}
    (\ts{T}_{\s}^{e} + \ts{T}_{\s}^{v} + \ts{T}_{\f}^{v})^{\circd{s}} \inversetranspose{\ts{F}}_{\s} \colondot \Grad{\bv{v}_{\s}}
\\
&\quad+
\int_{B_{\s}}
\bigl[
J_{\s} (\rho_{\f} \bv{a}_{\f}^{\circd{e}} 
- \bv{f}_{\f})^{\circd{s}} \cdot (\bv{v}_{\f}^{\circd{s}} - \bv{v}_{\s})
+ (\ts{T}_{\f}^{v})^{\circd{s}} \inversetranspose{\ts{F}}_{\s} \colondot \Grad{(\bv{v}_{\f}^{\circd{s}} - \bv{v}_{\s})}\bigr]
\\
&\quad+
\int_{\mathcal{S}_{\s}}
\llbracket
\rho_{\f}(\bv{v}_{\f}^{\circd{e}} - \bv{v}_{\s}^{\circd{e}}) \otimes (\bv{v}_{\f}^{\circd{e}} - \bv{v}_{\s}^{\circd{e}}) 
\rrbracket^{\circd{s}} J_{\s} \inversetranspose{\ts{F}}_{\s}  \bv{m}_{\s}
\cdot \bv{v}_{\s}
\\
&\quad-
\int_{\mathcal{S}_{\s}}
\biggl\llbracket
\Bigl(
k_{\f} \ts{I}
-
\rho_{\f} \bv{v}_{\f}^{\circd{e}} \otimes (\bv{v}_{\f}^{\circd{e}} - \bv{v}_{\s}^{\circd{e}})
\Bigr)^{\circd{s}} J_{\s} \inversetranspose{\ts{F}}_{\s} \bv{m}_{\s} \cdot (\bv{v}_{\f}^{\circd{s}} - \bv{v}_{\s})
\biggr\rrbracket
\\
&\quad
+
\int_{\mathcal{S}(t)}
\tfrac{1}{2} \mu_{\mathcal{S}}
\llbracket \bv{v}_{\flt} \rrbracket
\cdot
\llbracket \bv{v}_{\flt} \rrbracket.
\end{aligned}
\end{equation}
Substituting Eqs.~\eqref{eq: Ts and Tf overall} and~\eqref{eq: pf and ff full form} in the above expression, and using the definitions in Eqs.~\eqref{eq: Interior rate of dissipation} and~\eqref{eq: Surface Dissipation potential}, we have
\begin{equation}
\label{eq: Proof of energy estimate step 2}
\begin{aligned}[b]
0 &= \int_{B(t)}
\bigl(
    \rho_{\s} \bv{a}_{\s}^{\circd{e}} \cdot \bv{v}_{\s}^{\circd{e}}
    +
    \rho_{\f} \bv{a}_{\f}^{\circd{e}} \cdot \bv{v}_{\f}^{\circd{e}} \bigr)
    +
    \int_{B(t)}
    J_{\s}
    (\ts{T}_{\s}^{e})^{\circd{s}} \inversetranspose{\ts{F}}_{\s} \colondot \Grad{\bv{v}_{\s}}
\\
&\quad+
\int_{B(t)} \mathcal{D}^{\circd{s}} +
\int_{\mathcal{S}(t)}
\mathcal{D}_{\mathcal{S}(t)}.
\end{aligned}
\end{equation}
Using the results in Propositions~\ref{claim: acceleration kinetic energy relation} and~\ref{claim: TPE} in Appendix~\ref{Appendix B}, Eq.~\eqref{eq: formal stability estimate} follows.
\end{proof}

\begin{remark}[Consistency of the energy estimate]
    \label{remark: energy estimate}
    The energy estimate just proven is consistent with that in Eq.~\eqref{eq: TPE Mixture} when Eq.~\eqref{eq: dissipative interface case} is used and every supply of power to the system is suppressed. We note that, for any physically admissible formulation, and excluding the trivial case of rigid body motions, the last two terms on the left-hand side of Eq.~\eqref{eq: formal stability estimate} are positive for any motion of the fluid relative to the solid, thereby ensuring a monotonic decline of the sum of the kinetic, strain, and stored energies of the system.
\end{remark}

\begin{remark}[Analysis and \ac{FEM}]
    \label{remark: FEM implementation}
    We are not aware of a formal analysis of this specific problem articulated in the language of Sobolev spaces. Nor are we aware of an analysis within the context of the \ac{FEM}.  
    We have implemented Eq.~\eqref{eq: Abstract FEM mixture} using standard \ac{FEM} approaches typical of \ac{ALE} formulations. The interface $\mathcal{S}_{\s}$ was simply modeled in a conformal way, namely, it consisted of inter-element boundaries. The chosen interpolation functions are as follows:
    \begin{itemize}
        \item \nth{2} order Lagrange polynomials for the (globally continuous) displacement field.

        \item \nth{2} order Lagrange polynomials for the fluid velocity field. This field is not continuous across the interface.

        \item
        \nth{1} order Lagrange polynomials for the field $p$, which is not continuous across the interface.

        \item \nth{1} order Lagrange polynomials for the field $\ip$, which is defined only over the interface.
    \end{itemize}
    The appropriateness of this selection of finite element spaces is supported by a recent work by Barnafi \emph{et al.},\cite{Barnafi2021Mathematical-An} which, along with Ref.~\refcite{Bansal2024A-Lagrange-Mult}, offers some guidance for a rigorous analysis of the present work.
\end{remark}

\section{Results}
\label{Sec: Examples}
In this section, we present three sets of results. We begin by verifying our \ac{FEM} implementation using the method of manufactured solutions.\cite{Salari2000Code-Verification-0} We then present our analytical solution to a classic problem, to compare our results with existing ones. Finally, we solve an example problem to investigate the effect of the interface viscosity $\mu_\mathcal{S}$ on the jumps in the tangential component of filtration velocity and the pore pressure.

The numerical implementation of the formulation presented in this paper was done in \comsol,\cite{COMSOLCite} which has been used as a \ac{FEM} programming environment. We used a monolithic solver. For time integration we used the available variable time step and variable order \ac{BDF} method.\cite{Quarteroni2000Numerical-Mathematics-0} We have already mentioned the specific interpolation used. Additional details for the results presented will be indicated in context.

\subsection{Verification}
 To verify the formulation, we employed the method of manufactured solutions.\cite{Salari2000Code-Verification-0} We consider the planar geometry displayed in Fig.~\ref{fig:mms-problem-setup},
\begin{figure}[hbt]
    \centering
    \includegraphics[width=0.3\linewidth]{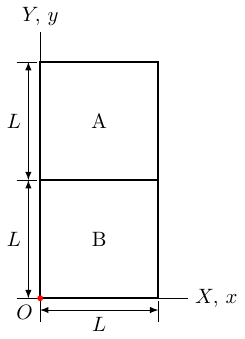}
    \caption{Two contiguous square poroelastic domains A and B used for the verification of the \ac{FEM} implementation of our formulation through the method of manufactured solutions.}
    \label{fig:mms-problem-setup}
\end{figure}
 which consists of two contiguous poroelastic domains with reference configurations A and B, respectively, aligned vertically and whose material properties are presented in Table~\ref{tab:mms-properties}.
\begin{table}[hbt]
\centering
\caption{Summary of parameter values used in the convergence rate study and in the last example.}
\label{tab:mms-properties}
\begin{tabular}{cccc}
\toprule
Quantity & Domain A & Domain B & Domains A \& B\\
\colrule
$\rho_s^*$          & & & \np[kg/m^3]{1e3}      \\
$\rho_f^*$          & & & \np[kg/m^3]{1e3}      \\
    $\mu_D$         & & & \np[Pa\cdot s]{1e-2}  \\
$\bar{\mu}_B$       & & & \np[Pa\cdot s]{1e-2}  \\
$\bar{\mu}_\f$      & & & \np[Pa\cdot s]{1e-2}  \\
$\mu_\mathcal{S}$   & & & \np[Pa\cdot s/m]{1e3} \\
$\phi_{R_{\s}}$     & $0.4$ & $0.8$ &           \\
$k_s$               & \np[m^2]{1e-9}
                    & \np[m^2]{1e-11}
                    &                           \\
$\mu_{R_e}$         & \np[Pa]{1e3}
                    & \np[Pa]{1e4}
                    &                           \\
\botrule
\end{tabular}
\end{table}
Referring to Fig.~\ref{fig:mms-problem-setup}, the position $\bv{X}$ of points in the reference configuration is expressed through the coordinates $X$ and $Y$. The position $\bv{x}$ of points in the deformed configuration is expressed through the coordinates $x = X + u_{\s_{X}}(X,Y,t)$ and $y = Y + u_{\s_{Y}}(X,Y,t)$, where $u_{\s_{X}}$ and $u_{\s_{Y}}$ are the horizontal and vertical components of the displacement field.
The images of A and B under the motion will be denoted by $\mathrm{A}(t)$ and $\mathrm{B}(t)$, respectively.
The interface's reference configuration $\mathcal{S}_{\s}$ is the segment with $X\in[0,L]$ and $Y = L$.  For both domains, the neo-Hookean hyperelastic material model $\hat{\Psi}_{\s}=\tfrac{1}{2} \mu_{R_e} \left(\bar{I}_1 - 3 \right)$ is chosen for the strain energy density of the pure solid phase, where $\bar{I}_1$ is the first principal invariant of $\bar{\ts{C}}_{\s} = J_{\s}^{-2/3} \ts{C}_{\s}$, namely the isochoric component of the right Cauchy-Green strain tensor $\ts{C}_{\s}$. $\mu_{R_e}$ is the referential solid's elastic shear modulus. Moreover, it is assumed that $\mu_\f=\bar{\mu}_\f \phi_\f^2$ and $\mu_B=\bar{\mu}_B \phi_\f^2$.

The manufactured solution consists of the following fields:
\begin{align}
    \label{eq: mms pA}
    p &= p_o \bigl(1-\cos(\pi t/t_o)\bigr) \cos\bigl[\pi(x+y)/L\bigr]
        & &\text{in $\mathrm{A}(t)$},
    \\
    \label{eq: mms pB}
    p &= p_o\bigl(1-\cos(\pi t/t_o)\bigr) \sin\bigl[\pi(x+y)/L\bigr]
        & &\text{in $\mathrm{B}(t)$},
    \\
    \label{eq: mms usXAB}
    u_{\s_{X}} &=
        u_o \bigl(1-\cos(\pi t/t_o)\bigr) \cos\bigl[\pi(X+Y)/L\bigr]
            & &\text{in A}\cup\text{B},\\
    \label{eq: mms usYAB}
    u_{\s_{Y}} &=
        u_o \bigl(1-\cos(\pi t/t_o)\bigr)
        \sin\bigl[\pi(X+Y)/L\bigr]
            & &\text{in A}\cup\text{B},
    \\
    \label{eq: mms vfe}
    \bv{v}_\f^{\circd{e}} &= (\partial_t \bv{u}_\s)^{\circd{e}} - (\phi_\f/\gamma) \Grad{p}
        & &\text{in $\mathrm{A}(t) \cup \mathrm{B}(t)$,}
    \\
    \label{eq: mms ipSt}
    \ip &= \ip_o \bigl(1-\cos(\pi t/t_o)\bigr) \cos(\pi x/L)
        & &\text{on $\mathcal{S}(t)$},
\end{align}
where $\gamma = (\mu_D \phi_\f^2)/\kappa_\s$, $u_o=\np[mm]{1}$, $p_o=\np[kPa]{2.5}$, $\ip_o=\np[kPa]{1}$, $L=\np[mm]{10}$, and $t_o=\np[s]{1}$.   These fields are substituted in the strong form of the governing equations to calculate the source terms in the interiors of their domains of definition as well as on the interface.

The initial conditions are those dictated by the manufactured solution for $t = 0$. For the solid part of the problem, the boundary conditions consist of prescribed displacements on the entirety of the external boundary; for the fluid part, traction boundary conditions are prescribed on the very top and bottom edges of the external boundary; the fluid's velocity is prescribed on the lateral surfaces. 

We have already indicated the choice of finite element spaces for our solutions. As far as the time integration part of the problem, we employed a \ac{BDF} method of order 2 with a constant time step of $\np[s]{0.001}$.

The selected manufactured solution in terms of solid displacement, fluid velocity, and pore pressure is depicted in Fig.~\ref{fig:mms-fields}
\begin{figure}[hbt]
    \centering
    \begin{subfigure}{0.2\textwidth}
        \centering
        \includegraphics[width=\linewidth]{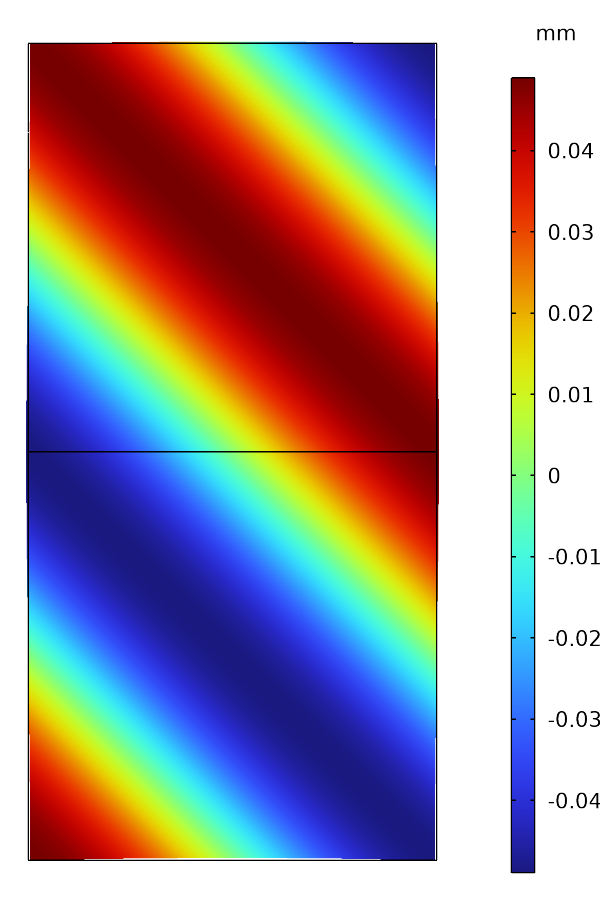}
        \caption{$\bv{u}_{\s_X}$.}
    \end{subfigure}
    \begin{subfigure}{0.2\textwidth}
        \centering
        \includegraphics[width=\linewidth]{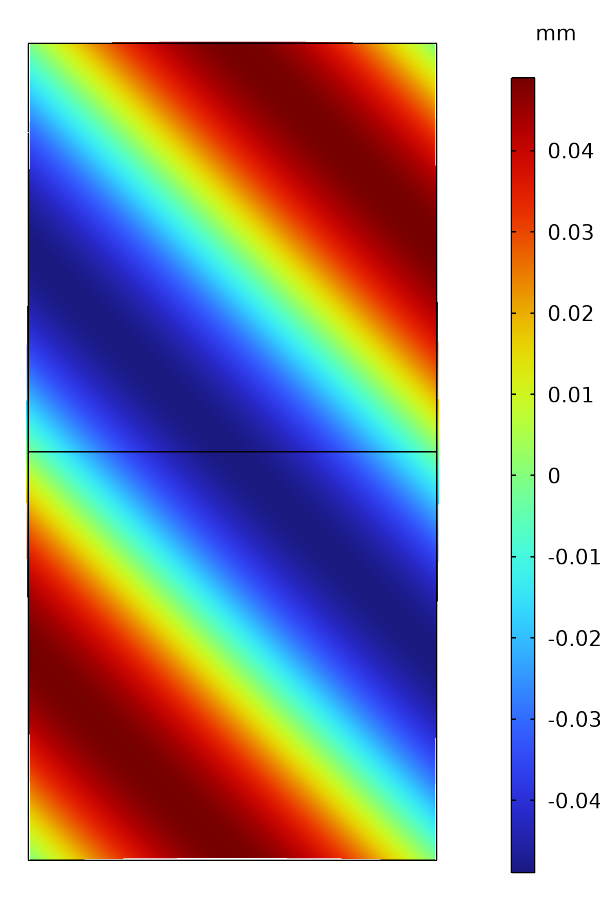}
        \caption{$\bv{u}_{\s_Y}$.}
    \end{subfigure}
    \\
    \begin{subfigure}{0.2\textwidth}
        \centering
        \includegraphics[width=\linewidth]{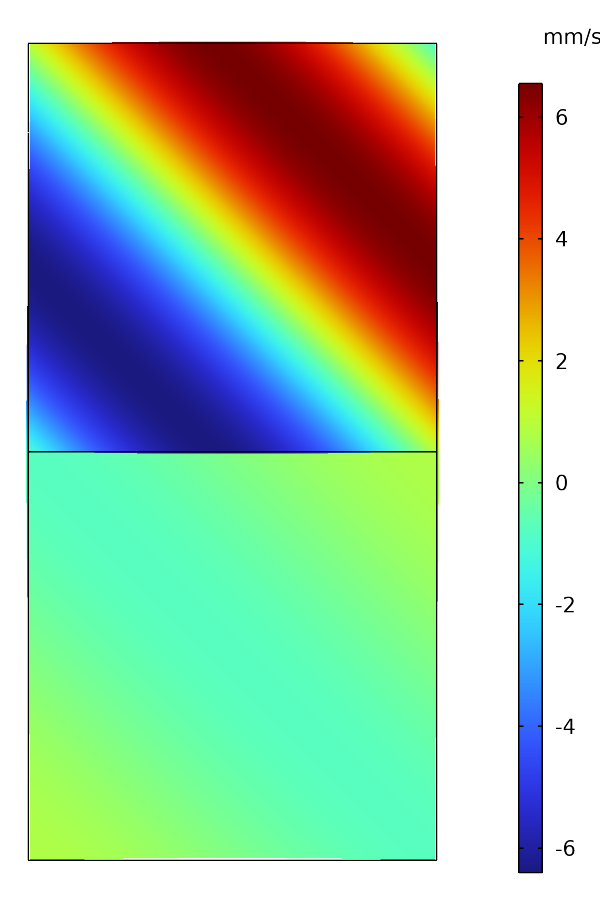}
        \caption{$\bv{v}_{\f_x}^{\circd{s}}$.}
    \end{subfigure}
    \begin{subfigure}{0.2\textwidth}
        \centering
        \includegraphics[width=\linewidth]{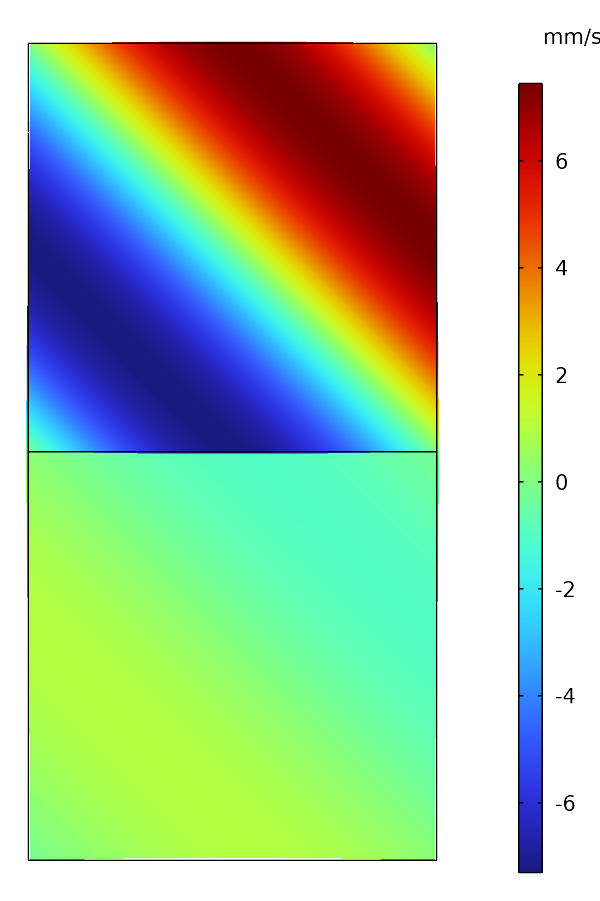}
        \caption{$\bv{v}_{\f_y}^{\circd{s}}$.}
    \end{subfigure}
    \begin{subfigure}{0.2\textwidth}
        \centering
        \includegraphics[width=\linewidth]{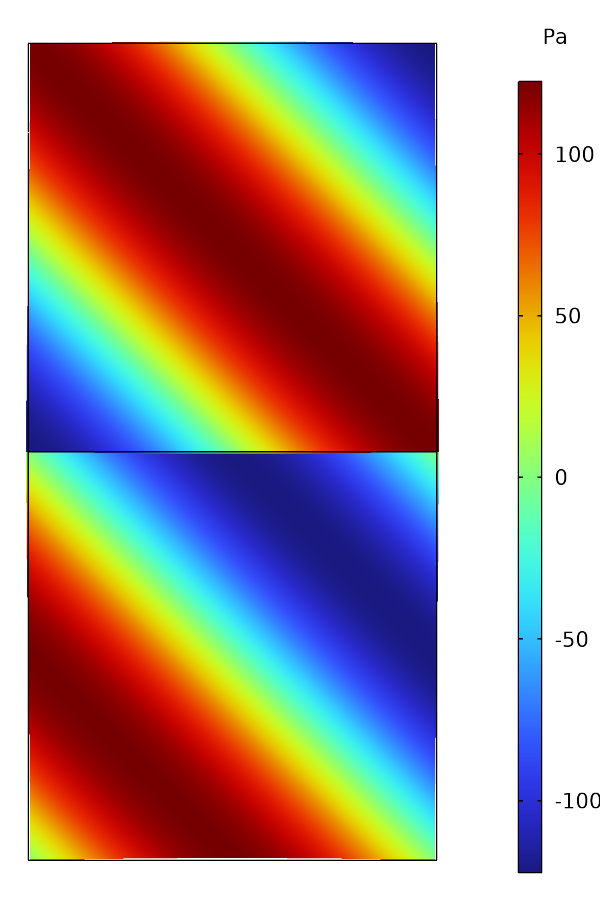}
        \caption{$p^{\circd{s}}$.}
    \end{subfigure}
    \caption{Panels (a) and (b) show the manufactured solutions for components of the displacement of the solid phase. Panels (c) and (d) show the components of the velocity of the fluid phase. Panel (e) shows the pore pressure. All plots are at time $t=\np[s]{0.1}$. For simplicity, all the plots are over the reference configuration of the solid phase.}
    \label{fig:mms-fields}
\end{figure}
at $t = \np[s]{0.1}$. This figure shows that indeed the selected solution is globally continuous in the displacement field, whereas it is discontinuous in the fluid velocity field and in the pore pressure field.

The mesh employed was made up of identical triangles and was uniformly refined through four cycles. The initial mesh diameter was $h = \np[mm]{1.25}$.

Figure~\ref{fig:mms-conv-rates}
\begin{figure}[hbt]
    \centering
    \begin{subfigure}{0.48\textwidth}
        \centering
        \includegraphics[width=\linewidth]{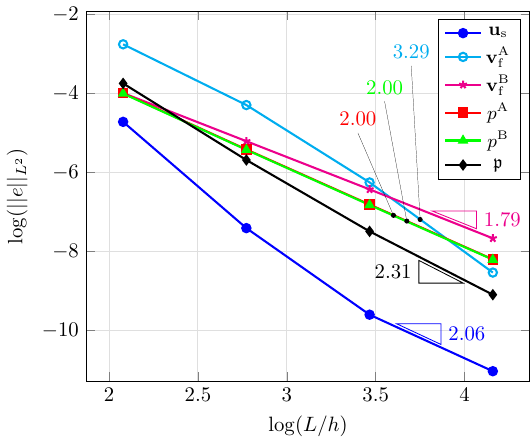}
        \caption{}
        \label{fig:mms-conv-rates-l2}
    \end{subfigure}
    \hfill
    \begin{subfigure}{0.48\textwidth}
        \centering
        \includegraphics[width=\linewidth]{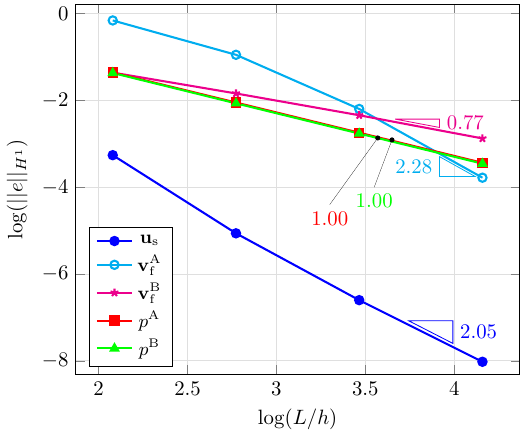}
        \caption{}
        \label{fig:mms-conv-rates-h1}
    \end{subfigure}
    \caption{Convergence rates for \subref{fig:mms-conv-rates-l2} the $L^2$-norm and \subref{fig:mms-conv-rates-h1} the $H^1$-semi-norm of the relative error $e$ between numerical and exact solution in all variables, at time $t=\np[s]{0.1}$, obtained using second-order Lagrange polynomials for the solid displacement and fluid velocity and first-order Lagrange polynomials for the pressure and the interface pressure. The computation of the norms of the relative errors is done via integrals that respect the domain of definition of said errors (e.g., the norms of the relative errors pertaining to $(\bv{v}_{\f}^{\mathrm{A}})^{\protect\circd{e}}$ are computed over $\mathrm{A}(t)$).
    }
    \label{fig:mms-conv-rates}
\end{figure}
presents the convergence rates under uniform refinement. Specifically, we report the $L^2$-norm and the $H^1$-semi-norm of the relative errors for each of the quantities specified in the manufactured solution. As the functions indicated in Eqs.~\eqref{eq: mms pA}--\eqref{eq: mms ipSt} each have their own domain of definition, the relative errors for these functions are computed via integrals that are consistent with said domains.

Provided that a formal analysis is not available so as to assess the optimality of the convergence rates, the convergence rates estimated here are rather good when compared to those typically obtained for the Navier-Stokes problem. We find this result rather encouraging, especially in view of the fact that we have relied on readily available numerical tools as opposed to devise custom-made algorithms for this problem.

\subsection{A classic problem}
To test our mathematical framework, we analytically solved a benchmark problem involving Poiseuille flow over a porous medium. In the Beavers-Joseph problem, a viscous fluid flows under a constant pressure gradient in a channel where the wall at the fluid-porous medium interface is permeable,\cite{Beavers1967Boundary-Condi,Discacciati2009Navier-Sto,Shim2022A-Hybrid-Biphas} as illustrated in Fig.~\ref{fig:poiseuille-flow}. 
\begin{figure}[hbt]
    \centering
    \includegraphics[width=0.7\linewidth]{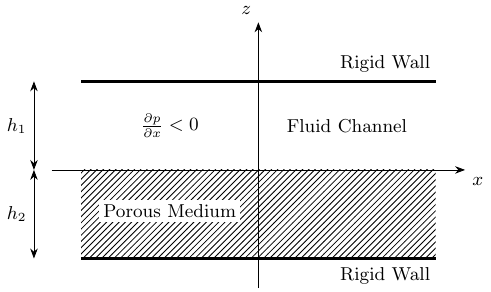}
    \caption{Schematic view of viscous fluid flow over a rigid porous medium, driven by a constant pressure gradient.}
    \label{fig:poiseuille-flow}
\end{figure}
The top and bottom boundaries of the system are assumed rigid and stationary, and the fluid velocity at these boundaries is set to zero. The flow inside the channel and the porous medium is only in the $x$ direction. Among others, this problem was previously considered by Hou \etal{}\cite{Hou1989Boundary-condit} in formulating boundary conditions at the interface between the pure fluid and the porous medium. In their study, they introduced a ``pseudo-no-slip'' kinematic interface condition based on the idea that the conditions at the interface between mixtures, or between mixtures and fluids, should reduce to those of single-phase continuum mechanics. They mentioned that, to formulate a well-posed mathematical problem for each subdomain of the fluid-solid interaction, it is necessary not only to satisfy the jump conditions for mass and momentum conservation, but also to impose additional kinematic constraints.
They chose to enforce the continuity of the pore pressure and the full continuity of the filtration velocity.

Here we solved the same problem but enforcing the jump conditions that were obtained in Eqs.~\eqref{eq: final abstract form of the jump condition of the fluid momentum}.

We begin by observing that the unidirectional nature of the fluid flow and the rigidity of the porous medium imply that the mass balance is trivially satisfied. The dimensionless form of the balance of linear momentum can then be derived for both the fluid channel and the porous medium, and is given by, respectively,
\begin{align}
    \dfrac{\partial^2 V_\mathrm{u}}{\partial Z^2}+1 = 0, \quad & \text{for $0 < Z < H$},
    \shortintertext{and}
    \delta^2 \dfrac{\partial^2 V_\mathrm{l}}{\partial Z^2} - V_\mathrm{l} + \eta \, \delta^2 = 0, \quad & \text{for $-1 < Z < 0$},
\end{align}
where the latter is the simplified version of Eq.~\eqref{eq: balance of momentum fluid from variation} and the former is the dimensionless form of the Navier-Stokes equation in the absence of inertia and body forces.
Here, $V_\mathrm{u}(Z)$ and $V_\mathrm{l}(Z)$ are the dimensionless fluid velocities in the channel and the porous medium, respectively. Additionally, $Z = z/h_{2}$, $H = h_{1}/h_{2}$, $\eta = \phi^{2}_{\f} \mu_{\f}/\mu_{\mathrm{a}}$, and 
$\delta = \sqrt{\mu_{\mathrm{a}}/(h_{2}^{2} \kappa_{\s})}$ are dimensionless parameters, where $\mu_{\mathrm{a}}$ represents the apparent viscosity of the fluid in the porous medium. Since the top and bottom boundaries are fixed and impermeable, the fluid velocity is zero at these boundaries. For the interface between the porous medium and the viscous fluid, the conditions represented in Eqs.~\eqref{eq: final abstract form of the jump condition of the fluid momentum} for the current problem can be simplified as
\begin{align}
    V^\prime_\mathrm{u} - \tfrac{1}{2}\gamma \left( V_\mathrm{u} - V_\mathrm{l} \right) = 0, \quad & \text{for} \quad Z \to 0^+,\\
    V^\prime_\mathrm{l} - \tfrac{1}{2}\gamma \eta \left( V_\mathrm{u} - V_\mathrm{l} \right) = 0, \quad & \text{for} \quad Z \to 0^-,
\end{align}
where $\gamma=\frac{\mu_\mathcal{S} h_2}{\mu_\f}$ is a dimensionless parameter representing the ratio of the interface viscosity to the fluid viscosity. Solving the governing equations along with the boundary conditions yields an analytical solution for the fluid velocity in both the channel and the porous medium.
\begin{figure}[ht!]
    \centering
    \begin{subfigure}{0.49\textwidth}
        \includegraphics[width=\linewidth]{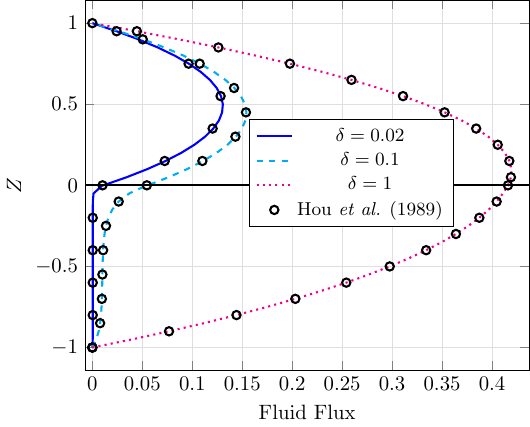}
        \caption{$H=1$}
        \label{fig:HouEtAl_H1_eta1}
    \end{subfigure}
    \hfill
    \begin{subfigure}{0.49\textwidth}
        \includegraphics[width=\linewidth]{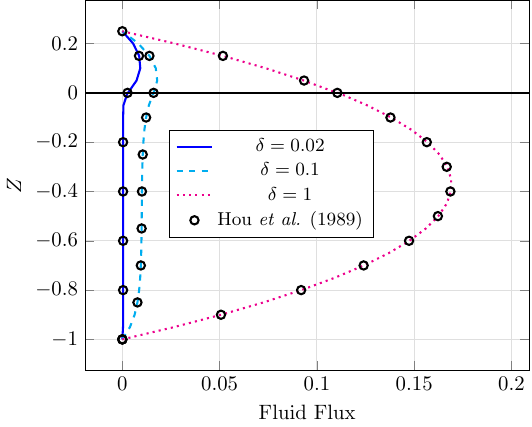}
        \caption{$H=0.25$}
        \label{fig:HouEtAl_H0.25_eta1}
    \end{subfigure}
    \caption{Comparison of profiles of fluid flux in the porous medium and the fluid channel with $\eta=1$ and $\gamma=1000$, for different values of $\delta$, and for (a) $H=1$ and (b) $H=0.25$. The lines represent results according to the formulation presented in this paper whereas the hollow circles represent the solution in Ref.~\protect\refcite{Hou1989Boundary-condit}.}
    \label{fig:HouEtAl_eta1}
\end{figure}

\begin{figure}[htb]
    \centering
    \begin{subfigure}{0.49\textwidth}
        \includegraphics[width=\linewidth]{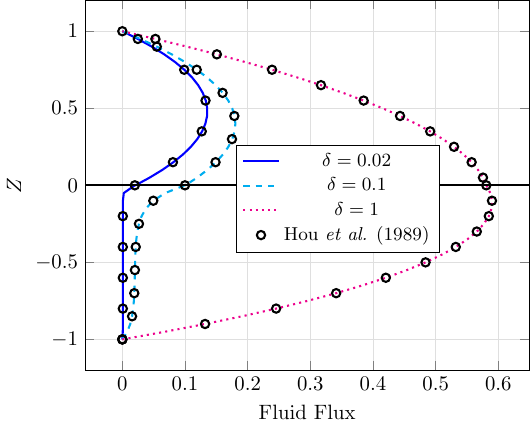}
        \caption{$H=1$}
        \label{fig:HouEtAl_H1_eta2}
    \end{subfigure}
    \hfill
    \begin{subfigure}{0.49\textwidth}
        \includegraphics[width=\linewidth]{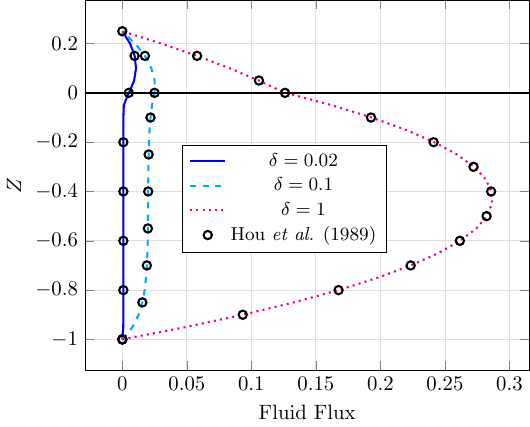}
        \caption{$H=0.25$}
        \label{fig:HouEtAl_H0.25_eta2}
    \end{subfigure}
    \caption{Comparison of profiles of fluid flux in the porous medium and the fluid channel with $\eta=2$ and $\gamma=1000$, for different values of $\delta$, and for (a) $H=1$ and (b) $H=0.25$. The lines represent results according to the formulation presented in this paper whereas the hollow circles represent the solution in Ref.~\protect\refcite{Hou1989Boundary-condit}.}
    \label{fig:HouEtAl_eta2}
\end{figure}
Figures~\ref{fig:HouEtAl_eta1} and~\ref{fig:HouEtAl_eta2} compare our analytical solution and the analytical solution in Hou \etal{}\cite{Hou1989Boundary-condit} Our solution is predicated on the use of a relatively large value of $\gamma$, which can be interpreted as representing a strongly dissipative interface. In this case, we show a strong agreement with the results reported by Hou \etal{}\cite{Hou1989Boundary-condit}, who had enforced continuity of both the pore pressure and the entire filtration velocity across the interface as the implementation of what these authors call a ``pseudo-no-slip'' interface condition. With this result we show that the pseudo-no-slip condition, which is a kinematic constraint, can be viewed as the manifestation of a constitutive assumption on the interface behavior. Furthermore, we show that it is not necessary the continuity of the pressure in a strict sense to obtain the above results.


\subsection{Role of interface viscosity}
To investigate the role of the interface viscosity, we set up a numerical example with the geometry shown in Fig.~\ref{fig:example-problem-setup}.
\begin{figure}[hbt]
    \centering
    \includegraphics[width=0.4\linewidth]{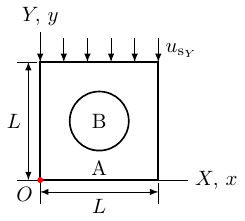}
    \caption{The example problem setup to study the effect of interface viscosity on the discontinuities of filtration velocity and pore pressure. The radius of domain B is $L/4$, where $L$ has the same value as used in the verification problem.}
    \label{fig:example-problem-setup}
\end{figure}
The parameters used are reported in Table~\ref{tab:mms-properties}, with the permeability of both subdomains increased by two orders of magnitude.

In this example problem, the top and bottom boundaries are impermeable, while the lateral boundaries are permeable and traction-free. The bottom boundary is fixed in the vertical direction, and a prescribed vertical displacement $u_{\s_{Y}} = u_o (t/t_o)^2$ is applied to the top boundary, with $u_o$ and $t_{o}$ having the same values used in the manufactured solution considered earlier. 
Both the top and bottom boundaries are free to move in the horizontal direction. We suppress the rigid body motions of the domain A. Moreover, we adopt quiescent initial conditions.

Figure~\ref{fig:jumps}\subref{fig:vflt_t_jump}
\begin{figure}[hbt]
    \centering
    \begin{subfigure}{0.49\textwidth}
        \centering
        \includegraphics[width=\linewidth]{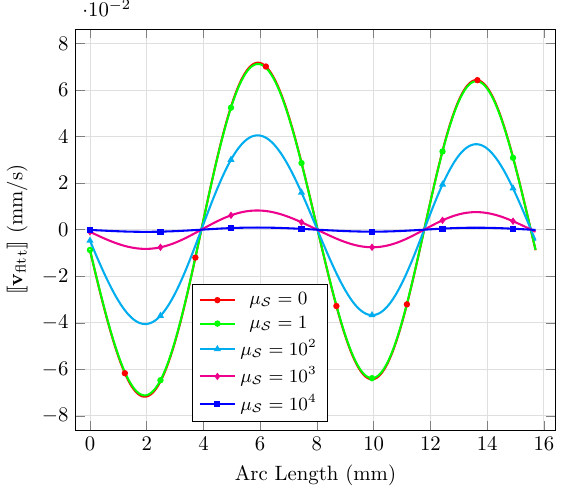}
        \caption{}
        \label{fig:vflt_t_jump}
    \end{subfigure}
    \hfill
    \begin{subfigure}{0.49\textwidth}
        \centering
        \includegraphics[width=\linewidth]{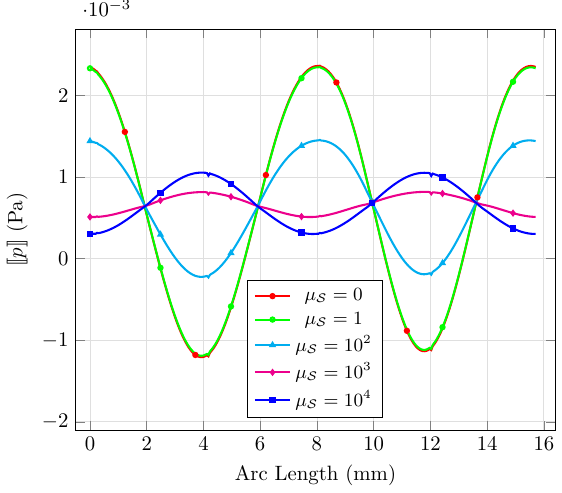}
        \caption{}
        \label{fig:p_jump}
    \end{subfigure}
    \caption{The effect of the interface viscosity $\mu_\mathcal{S}$ on \subref{fig:vflt_t_jump} the jump in the tangential component of the filtration velocity $\llbracket {\bv{v}_\flt}_\mathrm{t} \rrbracket$, and \subref{fig:p_jump} the jump in the pressure $\llbracket p \rrbracket$, at time $t=\np[s]{0.1}$. Both jump variables are plotted as a function of the interface arc length for several increasing values of $\mu_\mathcal{S}$.}
    \label{fig:jumps}
\end{figure}
shows the effect of interface viscosity on the jump in the tangential component of the filtration velocity. As the interface viscosity increases, the norm of the jump of the tangential component of the filtration velocity decreases.
Figure~\ref{fig:jumps}\subref{fig:p_jump} shows the jump in the pressure for different values of the interface viscosity. In this case, the jump in pressure does not seem to vanish. Furthermore, these observations are in agreement with the results in Fig.~\ref{fig:jumps_norms},
\begin{figure}[hbt]
    \centering
    \begin{subfigure}{0.477\textwidth}
        \centering
        \includegraphics[width=\linewidth]{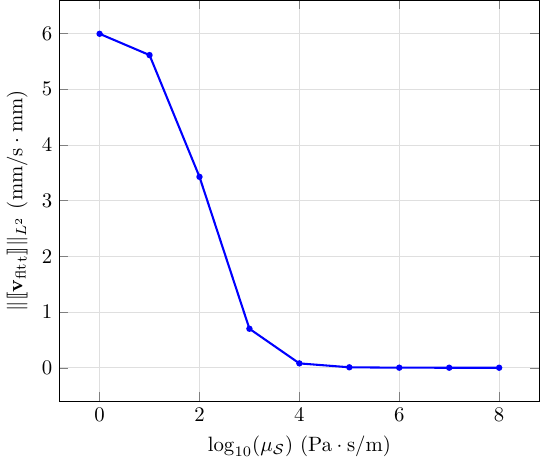}
        \caption{}
        \label{fig:vflt_t_jump_norm}
    \end{subfigure}
    \begin{subfigure}{0.5\textwidth}
        \centering
        \includegraphics[width=\linewidth]{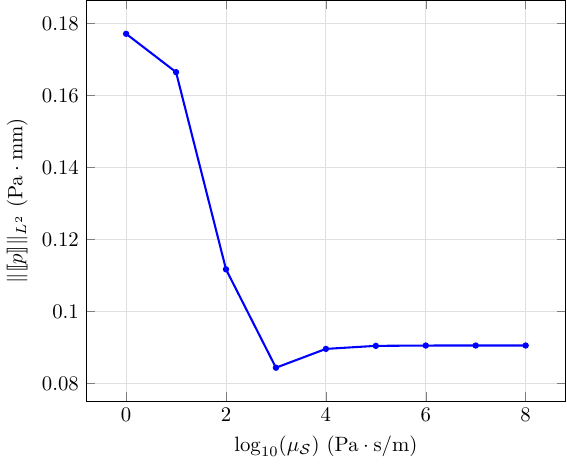}
        \caption{}
        \label{fig:p_jump_norm}
    \end{subfigure}
    \caption{The $L^2$-norm of the jump in \subref{fig:vflt_t_jump_norm} the tangential component of the filtration velocity $\llbracket {\bv{v}_\flt}_\mathrm{t} \rrbracket$, and \subref{fig:p_jump_norm} the pressure $\llbracket p \rrbracket$, as a function of the interface viscosity $\mu_\mathcal{S}$, at time $t=\np[s]{0.1}$.}
    \label{fig:jumps_norms}
\end{figure}
displaying the $L^2$-norm of the two jumps as a function of the interface viscosity $\mu_\mathcal{S}$. Again, as the interface viscosity increases, the jump in the tangential component of the filtration velocity tends to zero, while the jump in the pressure converges to a non-zero value.

\section{Summary, discussion, and conclusions}
\label{Sec: Summary}
In this work we have proposed an approach to a variational formulation of the conditions at the interface separating two contiguous poroelastic domains. The formulation stems from mixture theory. We focus our attention to mixture of constituents that are incompressible in their pure form. This said, the present work is applicable to problems with large deformations. Our approach follows very closely that by dell'Isola \etal{}\cite{dellIsola2009Boundary-Condit-0} and adapts it to the constitutive class considered herein. We also offer a connection to a treatment of the subject according to a traditional continuum mechanics approach.\cite{Bowen1976Theory-of-Mixtu0}

From a modeling perspective, our first goal was to make a more direct link between the constitutive assumptions pertaining to the behavior of the interface and the corresponding interface conditions. This is in contrast to interface conditions that rely on kinematic constraints such as the full continuity of the filtration velocity and/or of the pore pressure. The second goal of our work was to produce a variational formulation of the problem easily implementable as a \ac{FEM} able to respect, at least formally, the energetics of the problem. In such an attempt, we chose test functions for the fluid velocity to have arbitrary jumps at the interface, while the constraint originating with the jump condition of the balance of mass is imposed by way of a Lagrange multiplier.

The results from the examples we have included are encouraging in that we have empirically shown convergence for our formulation using the method of manufactured solutions. In doing so, we have identified what seems to be a practical choice of stable finite element spaces.

We have also solved a classic problem of flow through a rigid porous medium showing that we can recover some well-established results by controlling the constitutive response of the interface. In our case, the formulation contains a dissipative term associated to the norm of the jump of the filtration velocity. This term can be interpreted as a penalization of the jump of the filtration velocity. As the value of the parameter $\mu_{\mathcal{S}}$ is increased, our solution appears to approach a solution with continuous filtration velocity. The last set of results confirms this behavior, this time in a full poroelastic setting. However, it also shows that the jump of the pore pressure field across the interface is not directly controlled by the parameter that controls the jump of the filtration velocity. This conclusion was to be expected as these fields are indeed independent.

In conclusion, we believe that the proposed formulation has some encouraging results that warrant further development and analysis. As mentioned in the introduction, our interest lies in brain physiology and the modeling of membranes in the central nervous system and we plan future applications in this arena. These developments will require more sophisticated models to regulate the pore pressure jump across the interfaces under consideration. Another line of inquiry concerns the adaptation of the proposed \ac{FEM} formulation to more general mixtures, say with compressible constituents, and with the kinematics considered in Ref.~\refcite{dellIsola2009Boundary-Condit-0}. From a numerical perspective, a formal investigation on stability issues associated to the compatibility of spaces of the Lagrange multipliers in relation to the other fields is necessary. An important area of study is also stability relative to constitutive parameters as discussed, for example, by Rodrigo \etal\cite{Rodrigo2018New-Stabilized-}

\appendix
\renewcommand\thesection{\Alph{section}}
\makeatletter
\def\@seccntformat#1{\@ifundefined{#1@cntformat}%
   {{\upshape{\csname the#1\endcsname.}}\hskip .5em}%
   {\csname #1@cntformat\endcsname}}
\newcommand\section@cntformat{{\upshape{\appendixname\ \thesection.}}\hskip .5em}
\makeatother

\section{Notational conventions}
\label{Appendix: Notation}
\subsection{Preliminaries}
$\mathbb{R}$ will denote the real number set, with $\mathbb{R}^{+}$ and $\mathbb{R}^{+}_{0}$ denoting its positive and non-negative subsets, respectively. $\mathcal{E}^{3}$ will denote the three-dimensional Euclidean point space and $t\in \mathbb{R}^{+}_{0}$ time. The translation (vector) space of $\mathcal{E}^{3}$ will be denoted by $\mathscr{V}$. Given vector spaces $\mathscr{U}$ and $\mathscr{W}$, the vector space of linear operators with domain $\mathscr{U}$ and range in $\mathscr{W}$  will be denoted by $\Lin[\mathscr{U}]{\mathscr{W}}$. The expression $\Lin{\mathscr{V}}$ is shorthand for $\Lin[\mathscr{V}]{\mathscr{V}}$.

Physical quantities are described by means of functions whose domain is a specific manifold. Often, the physical meaning of a quantity remains unchanged when the function representing it is composed with a map between a manifold and another. This fact is at the root of the distinction between Eulerian, Lagrangian, and Eulerian--Lagrangian frameworks (cf.\ Refs.~\refcite{GurtinFried_2010_The-Mechanics_0} and~\refcite{Fernandez2009The-Derivation-0}) in Continuum Mechanics. To indicate the domain of a function representing one of these physical quantities, we adopt the shorthand notation in Ref.~\refcite{dellIsola2009Boundary-Condit-0}. For any $t$, let $\Omega_{a}(t) \subset \mathcal{E}^{3}$ and $\Omega_{b}(t) \subset \mathcal{E}^{3}$ be two possibly time-dependent smooth manifolds whose points are denoted by $\bv{X}_{a}$ and $\bv{X}_{b}$, respectively. Let $\bv{\chi}_{t}: \Omega_{a}(t) \to\Omega_{b}(t)$ such that $\bv{\chi}_{t}(\bv{X}_{a}) = \bv{X}_{b}\in\Omega_{b}(t)$ be a homeomorphism. Let $\phi:\Omega_{a}(t)\to\mathscr{U}$ and $\psi: \Omega_{b}(t)\to\mathscr{W}$ be continuous functions  taking values in the vector spaces $\mathscr{U}$ and $\mathscr{W}$, respectively. Then, we denote by $\phi^{\circd{b}}$ and $\psi^{\circd{a}}$ the following functions
\begin{equation}
\label{eq: circd notation}
\phi^{\circd{b}} \coloneqq \phi \circ \bv{\chi}_{t}^{-1}
\quad\text{and}\quad
\psi^{\circd{a}}  \coloneqq \psi \circ \bv{\chi}_{t}.
\end{equation}

As in Ref.~\refcite{GurtinFried_2010_The-Mechanics_0}, we view second-order tensors as linear operators from a vector space to another. Given a second-order tensor $\ts{A}$, its transpose will be denoted by $\transpose{\ts{A}}$. If $\ts{A}$ is invertible, its inverse and its inverse transpose will be denoted by $\ts{A}^{-1}$ and $\inversetranspose{\ts{A}}$, respectively. Given a vector-valued quantity $\bv{u}$ in the domain of $\ts{A}$, the action of $\ts{A}$ on $\bv{u}$ will be denoted by $\ts{A}\bv{u}$.

Given a smooth manifold $\mathcal{A}$ and a point $P\in\mathcal{A}$, $\mathcal{T}_{P}\mathcal{A}$ denotes the tangent space at $P$. Clearly, in an Euclidean context, if $\mathcal{A}$ is a three-dimensional submanifold of $\mathcal{E}^{3}$, its tangent spaces are isomorphic to $\mathscr{V}$. If a metric is defined over $\mathcal{A}$, the notation $\Grad\psi$ will denote the gradient of the field $\psi$. We note that a gradient is not a stand-alone object in that it is inherently linked to the metric manifold over which it is defined\,---\,a gradient on a manifold is different from that on another manifold. The notation $\Div{\square}$ will denote the divergence of the quantity $\square$, assuming that $\square$ admits a divergence. As is the convention in elementary calculus, we will use the notation $\partial_{t}\square$ to denote the derivative of $\square$ with respect to $t$ while holding all other arguments of $\square$ fixed.

Let a smooth surface $\Sigma$ be a codimension-1 submanifold of a three-dimensional $\mathcal{A}\subset\mathcal{E}^{3}$. Let $\bv{\nu}$ be a unit vector field orienting $\Sigma$. Let $\phi(\bv{X})$, $\bv{X}\in\mathcal{A}$, be some field that is continuous away from $\Sigma$. Then, at $\hat{\bv{X}}\in\Sigma$,
\begin{equation}
\label{eq: surface limits def}
\phi^{\pm} \coloneqq
\lim_{s\to0^{+}} \phi(\hat{\bv{X}}\pm s \bv{\nu}),\quad
\llbracket\phi\rrbracket \coloneqq \phi^{+} - \phi^{-},
\quad\text{and}\quad
\langle\phi\rangle \coloneqq \tfrac{1}{2}(\phi^{+}+\phi^{-}).
\end{equation}

\section{}
\label{Appendix B}
Here we provide explicit derivations for some relationships that were referenced in the main text.

\begin{proposition}
\label{claim: acceleration kinetic energy relation}
Under the assumptions stated in the main paper, we have
\begin{equation}
\label{eq: power of material acceleration}
\bv{v}_{a} \cdot \rho_{a} \bv{a}_{a} = \partial_{t} \left(\tfrac{1}{2} \rho_{a} \|\bv{v}_{a}\|^{2}\right)
+
\Div{
\left(
\tfrac{1}{2} \rho_{a} \|\bv{v}_{a}\|^{2} \bv{v}_{a}
\right)
}.
\end{equation}
\end{proposition}
\begin{proof}
We recall that $\bv{a}_{a}$ is the material acceleration of species $a$. Hence, in Eulerian form we have
\begin{equation*}
\bv{a}_{a} = \partial_{t}\bv{v}_{a} + \ts{L}_{a}\bv{v}_{a} = \partial_{t}\bv{v}_{a} + \Grad{\bv{v}_{a}}[\bv{v}_{a}],
\end{equation*}
where we recall that $\ts{L}_{a} = \Grad{\bv{v}_{a}}$ and we write $\Grad{\square}[\diamond]$ to denote the action of the gradient of a field $\square$ onto the field $\diamond$, assuming that the operation is meaningful for $\square$ and $\diamond$. We also recall the following two identities (cf.\ Ref.~\refcite{GurtinFried_2010_The-Mechanics_0}):
\begin{enumerate}[i]
\item
Given smooth vector fields $\bv{w}$ and $\bv{u}$,
\begin{equation}
\label{eq: appx identity grad dot}
\Grad{(\bv{w}\cdot\bv{u})} = \transpose{(\Grad{\bv{w}})}[\bv{u}] + \transpose{(\Grad{\bv{u}})}[\bv{w}].
\end{equation}
\item Given a smooth scalar field $\phi$ and a smooth vector field $\bv{w}$,
\begin{equation}
\label{eq: appx identity div prod scalar vec}
\Div{(\phi \bv{w})} = \Grad{\phi}[\bv{w}] + \phi \Div{\bv{w}}.
\end{equation}
\end{enumerate}
Finally, we recall that the balance of mass for species $a$ in the absence of chemical reaction is
\begin{equation*}
\partial_{t} \rho_{a} + \Div{(\rho_{a}\bv{v}_{a})} = 0
\quad\Rightarrow\quad
\partial_{t} \rho_{a} = -\Div{(\rho_{a}\bv{v}_{a})}.
\end{equation*}
With the above in mind, we have 
\begin{equation*}
\begin{aligned}
\bv{v}_{a} \cdot \rho_{a} \bv{a}_{a}
&= \rho_{a}(\bv{v}_{a}\cdot\partial_{t}\bv{v}_{a} + \bv{v}_{a} \cdot \ts{L}_{a} \bv{v}_{a})
\\
&= \rho_{a}(\bv{v}_{a}\cdot\partial_{t}\bv{v}_{a} + \transpose{\ts{L}_{a}}\bv{v}_{a} \cdot \bv{v}_{a})
\\
&=
\tfrac{1}{2} \rho_{a}
\left(
\partial_{t} \|\bv{v}_{a}\|^{2}
+
\grad{\|\bv{v}_{a}\|^{2}}[\bv{v}_{a}]
\right)
\\
&=
\partial_{t} \left(\tfrac{1}{2} \rho_{a} \|\bv{v}_{a}\|^{2}\right)
+
\tfrac{1}{2} \rho_{a} \grad{\|\bv{v}_{a}\|^{2}}[\bv{v}_{a}]
-(\partial_{t}\rho_{a}) \tfrac{1}{2}\|\bv{v}_{a}\|^{2}
\\
&=
\partial_{t} \left(\tfrac{1}{2} \rho_{a} \|\bv{v}_{a}\|^{2}\right)
+
\tfrac{1}{2} \rho_{a} \grad{\|\bv{v}_{a}\|^{2}}[\bv{v}_{a}]
+
\left(
\Div{\rho_{a}\bv{v}_{a}}
\right)
\tfrac{1}{2}\|\bv{v}_{a}\|^{2}
\\
&=
\partial_{t} \left(\tfrac{1}{2} \rho_{a} \|\bv{v}_{a}\|^{2}\right)
+
\Div{
\left(
\tfrac{1}{2} \rho_{a} \|\bv{v}_{a}\|^{2} \bv{v}_{a}
\right)
}.
\end{aligned}
\end{equation*}
\end{proof}

\begin{proposition}
Let $\Omega$ be the control volume defined in Section~\ref{Sec: TPE}. Recalling the definitions in Section~\ref{sec: CRF}, we have
\begin{equation}
\label{eq: elastic stresses of solid}
\int_{\Omega} \ts{T}_{\s}^{e} \colondot \ts{L}_{\s}
=
\frac{\d{}}{\d{t}} \int_{\Omega} \psi_{\s} + \int_{\partial\Omega_{\text{\emph{ext}}}} \psi_{\s} \bv{v}_{\s}^{\circd{\emph{e}}} \cdot \bv{n}.
\end{equation}
\end{proposition}
\begin{proof}
    Let $\Omega_{\s}$ be the inverse image of $\Omega$ under the solid's motion $\bv{\chi}_{\s}$. We recall that $\Omega$ is partitioned into two domains by the moving surface $\mathcal{S}(t)$ and that $\partial\Omega_{\text{ext}}$, the external boundary of $\Omega$, is fixed. By contrast, in the solid's reference configuration, the inverse image of $\mathcal{S}(t)$, namely $\mathcal{S}_{\s}$, is fixed, whereas $\partial\Omega_{\s_{\text{ext}}}$, the inverse image of $\partial\Omega_{\text{ext}}$, is a moving surface with normal velocity $-\ts{F}_{\s}^{-1}\bv{v}_{\s}\cdot\bv{n}_{\s}$, where $\bv{n}_{\s}$ is the unit normal orienting $\partial\Omega_{\s_{\text{ext}}}$, outward relative to $\Omega_{\s}$. In fact, let $\hat{\bv{X}}(t)$ be a point on $\partial\Omega_{\s_{\text{ext}}}$. This point is convected by the solid's motion to the \emph{fixed} point $\bv{\chi}_{\s}(\hat{\bv{X}}(t),t)$ on $\partial\Omega_{\text{ext}}$ so that
\begin{equation*}
\frac{\d{}}{\d{t}} \bv{\chi}_{\s}(\hat{\bv{X}}(t),t) = \bv{0}
\quad\Rightarrow\quad
\partial_{t}\hat{\bv{X}} = -\ts{F}_{\s}^{-1} \bv{v}_{\s}.
\end{equation*}
With the above in mind, we have
\begin{equation}
\label{eq: ddt integral of psi}
\begin{aligned}[b]
    \frac{\d{}}{\d{t}} \int_{\Omega} \psi_{\s} &= 
    \frac{\d{}}{\d{t}} \int_{\Omega_{\s}} J_{\s} \psi_{\s}^{\circd{s}}
\\
&=
\int_{\Omega_{\s}} \partial_{t} \Psi_{\s} -
\int_{\partial\Omega_{{\s}_{\text{ext}}}} J_{\s} \psi_{\s}^{\circd{s}} \ts{F}_{\s}^{-1}\bv{v}_{\s} \cdot \bv{n}_{\s}
\\
&=
\int_{\Omega_{\s}} \partial_{t} \Psi_{\s} -
\int_{\partial\Omega_{{\s}_{\text{ext}}}}  \psi_{\s}^{\circd{s}} \bv{v}_{\s} \cdot (J_{\s} \inversetranspose{\ts{F}}_{\s} \bv{n}_{\s}).
\end{aligned}
\end{equation}
Next, making use of the second of Eqs.~\eqref{eq: 1PK and Cauchy stress}, we have
\begin{equation*}
\begin{aligned}
    \partial_{t}\Psi_{\s} &= \frac{\partial\Psi_{\s}}{\partial\ts{C}_{\s}} \colondot \partial_{t}\ts{C}_{\s}
\\
&=
\frac{1}{2} J_{\s} \ts{F}_{\s}^{-1} (\ts{T}_{\s}^{e})^{\circd{s}} \inversetranspose{\ts{F}}_{\s} \colondot \partial_{t}\ts{C}_{\s}.
\end{aligned}
\end{equation*}
Next, we recall that $\ts{C}_{\s} = \transpose{\ts{F}}_{\s}\ts{F}_{\s}$, so that $\partial_{t}\ts{C}_{\s} = 2 \sym{(\transpose{\ts{F}}_{\s}\partial_{t}\ts{F}_{\s})}$. Hence, taking advantage of the symmetry of $\ts{T}_{\s}^{e}$, we can then write
\begin{equation}
\label{eq: power of elastic stress transformations}
\begin{aligned}
    \partial_{t}\Psi_{\s} &= 
J_{\s} \ts{F}_{\s}^{-1} (\ts{T}_{\s}^{e})^{\circd{s}} \inversetranspose{\ts{F}}_{\s} \colondot \transpose{\ts{F}}_{\s}\partial_{t}\ts{F}_{\s}
\\
&=
J_{\s} (\ts{T}_{\s}^{e})^{\circd{s}} \colondot (\partial_{t}\ts{F}_{\s}) \ts{F}_{\s}^{-1}
\\
&= 
J_{\s} (\ts{T}_{\s}^{e})^{\circd{s}} \colondot \ts{L}_{\s}^{\circd{s}}.
\end{aligned}
\end{equation}
Substituting Eq.~\eqref{eq: power of elastic stress transformations} into Eq.~\eqref{eq: ddt integral of psi}, pushing the integration to the domain $\Omega$, and rearranging terms, Eq.~\eqref{eq: elastic stresses of solid} follows.
\end{proof}

\begin{proposition}
\label{claim: TPE}
    Equation~\eqref{eq: TPE Mixture} holds true.
\end{proposition}
\begin{proof}
Summing Eqs.~\eqref{eq: TPW inverse step 1} for $a = \f$ and $a = \s$, we have
\begin{equation}
\label{eq: sum of  TPW inverse step 1}
\begin{multlined}
\frac{\d{}}{\d{t}} \int_{\Omega}(k_{\f} + k_{\s})
+
\int_{\Omega} (\ts{T}_{\f}\colondot \ts{L}_{\f} + \ts{T}_{\s}\colondot \ts{L}_{\s} - \bv{p}_{\f} \cdot (\bv{v}_{\f}^{\circd{e}} - \bv{v}_{\s}^{\circd{e}}))
\\
= -\int_{\partial\Omega_{\text{ext}}} (k_{\f} \bv{v}_{\f}^{\circd{e}} + k_{\s} \bv{v}_{\s}^{\circd{e}} ) \cdot \bv{n}
+
\int_{\partial\Omega_{\text{ext}}}(
\ts{T}_{\f}\bv{n} \cdot \bv{v}_{\f}^{\circd{e}}
+
\ts{T}_{\s}\bv{n} \cdot \bv{v}_{\s}^{\circd{e}}
)
\\
+
\int_{\mathcal{S}(t)}
\llbracket
k_{\f} (\bv{v}_{\f}^{\circd{e}}-\bv{v}_{\s}^{\circd{e}}) \cdot \bv{m}
- 
\ts{T}_{\f} \bv{m} \cdot \bv{v}_{\f}^{\circd{e}}
- 
\ts{T}_{\s} \bv{m} \cdot \bv{v}_{\s}^{\circd{e}}
\rrbracket
\\
+
\int_{\Omega} (\bv{b}_{\f} \cdot \bv{v}_{\f}^{\circd{e}} + \bv{b}_{\s} \cdot \bv{v}_{\s}^{\circd{e}}),
\end{multlined}
\end{equation}
where we have used the property that $\bv{p}_{\s} = -\bv{p}_{\f}$. Using Eqs.~\eqref{eq: pf and ff full form}, for the last term on the left-hand side of Eq.~\eqref{eq: sum of  TPW inverse step 1}, we have
\begin{equation}
\label{eq: p term}
\begin{aligned}
-\bv{p}_{\f} \cdot (\bv{v}_{\f}^{\circd{e}} - \bv{v}_{\s}^{\circd{e}}) &= -p \Grad{\phi_{\f}} \cdot (\bv{v}_{\f}^{\circd{e}} - \bv{v}_{\s}^{\circd{e}}) - \bv{f}_{\f} \cdot (\bv{v}_{\f}^{\circd{e}} - \bv{v}_{\s}^{\circd{e}})
\\
&=
-p \Div{[\phi_{\f} (\bv{v}_{\f}^{\circd{e}} - \bv{v}_{\s}^{\circd{e}})]}
+ p \phi_{\f} \Div{(\bv{v}_{\f}^{\circd{e}} - \bv{v}_{\s}^{\circd{e}})}
\\
&\qquad-\bv{f}_{\f} \cdot (\bv{v}_{\f}^{\circd{e}} - \bv{v}_{\s}^{\circd{e}}).
\end{aligned}
\end{equation}
In view of Eqs.~\eqref{eq: 1PK and Cauchy stress}--\eqref{eq: Ts and Tf overall}, and recalling that the Cauchy stresses in these equations are symmetric, we also have
\begin{equation}
\label{eq: stress power interior}
\begin{aligned}
\ts{T}_{\f}\colondot\ts{L}_{\f}
+
\ts{T}_{\s}\colondot\ts{L}_{\s}
&=
-p \phi_{\f} \, \underbrace{\ts{I} \colondot \ts{L}_{\f}}_{\Div{\,\bv{v}_{\f}^{\circd{e}}}}
+ 2 \mu_{\f} \ts{D}_{\f}\colondot\ts{D}_{\f} + 2 \mu_{B} (\ts{D}_{\f} - \ts{D}_{\s}) \colondot \ts{D}_{\f}
\\
&\qquad
-p \phi_{\s} \, \underbrace{\ts{I} \colondot \ts{L}_{\s}}_{\Div{\,\bv{v}_{\s}^{\circd{e}}}}
+ \ts{T}_{\s}^{e} \colondot \ts{L}_{\s}
- 2 \mu_{B}(\ts{D}_{\f} - \ts{D}_{\s}) \colondot \ts{D}_{\s}
\\
&=
-p \phi_{\f} \Div{(\bv{v}_{\f}^{\circd{e}} - \bv{v}_{\s}^{\circd{e}})} - p \Div{\bv{v}_{\s}^{\circd{e}}}
\\
&\qquad + 2 \mu_{\f}\|\ts{D}_{\f}\|^{2}
+ 2 \mu_{B} \|\ts{D}_{\f} - \ts{D}_{\s}\|^{2}
\\
&\qquad + \ts{T}_{\s}^{e} \colondot \ts{L}_{\s}.
\end{aligned}
\end{equation}
Summing Eqs.~\eqref{eq: p term} and~\eqref{eq: stress power interior}, and using Eqs.~\eqref{eq: saturation condition} along with the second of Eqs.~\eqref{eq: div vvavg constraint expanded}, we then have that
\begin{equation}
\label{eq: interior dissipation terms}
\ts{T}_{\f}\colondot\ts{L}_{\f}
+
\ts{T}_{\s}\colondot\ts{L}_{\s} -\bv{p}_{\f} \cdot (\bv{v}_{\f}^{\circd{e}} - \bv{v}_{\s}^{\circd{e}}) = \ts{T}_{\s}^{e} \colondot \ts{L}_{\s} + \mathscr{D},
\end{equation}
where we have used Eqs.~\eqref{eq: pf and ff full form} and the definition of $\mathscr{D}$ in Eq.~\eqref{eq: Interior rate of dissipation}. Integrating Eq.~\eqref{eq: interior dissipation terms} over $\Omega$, and using Eq.~\eqref{eq: elastic stresses of solid}, we can now write
\begin{equation}
\label{eq: interior dissipation terms expanded}
\int_{\Omega}(
\ts{T}_{\f}\colondot\ts{L}_{\f}
+
\ts{T}_{\s}\colondot\ts{L}_{\s} -\bv{p}_{\f} \cdot (\bv{v}_{\f}^{\circd{e}} - \bv{v}_{\s}^{\circd{e}}) ) = 
\frac{\d{}}{\d{t}} \int_{\Omega} \psi_{\s} + \int_{\partial\Omega_{\text{ext}}} \psi_{\s} \bv{v}_{\s}^{\circd{e}} \cdot \bv{n}.
+\int_{\Omega} \mathscr{D}.
\end{equation}
Using the definition in Eq.~\eqref{eq: Externally supplied power}, Eq.~\eqref{eq: sum of  TPW inverse step 1} can now be rewritten as
\begin{multline}
\label{eq: TPE almost done}    
\mathcal{W}_{\text{ext}}(\Omega)
=
\frac{\d{}}{\d{t}} \int_{\Omega} (k_{\s} + k_{\f} + \psi_{\s}) + \int_{\Omega} \mathscr{D}
\\
-
\int_{\mathcal{S}(t)}
\llbracket
k_{\f} (\bv{v}_{\f}^{\circd{e}}-\bv{v}_{\s}^{\circd{e}}) \cdot \bv{m}
- 
\ts{T}_{\f} \bv{m} \cdot \bv{v}_{\f}^{\circd{e}}
- 
\ts{T}_{\s} \bv{m} \cdot \bv{v}_{\s}^{\circd{e}}
\rrbracket.
\end{multline}
We observe that, using the definition in Eq.~\eqref{eq: d def}, and recalling that both $d$ and $\bv{v}_{\s}$ are continuous across $\mathcal{S}(t)$,
\begin{equation}
\label{eq: jump kf}
\begin{aligned}[b]
\llbracket
k_{\f} (\bv{v}_{\f}^{\circd{e}}-\bv{v}_{\s}^{\circd{e}}) \cdot \bv{m} 
\rrbracket
&=
\llbracket
\tfrac{1}{2} d \|\bv{v}_{\f}^{\circd{e}}\|^{2}
\rrbracket
\\
&=
d
\llbracket
\tfrac{1}{2} \|\bv{v}_{\f}^{\circd{e}}\|^{2}
\rrbracket
\\
&=
d
\llbracket
\tfrac{1}{2}
\|\bv{v}_{\f}^{\circd{e}} - \bv{v}_{\s}^{\circd{e}}\|^{2}
+
\bv{v}_{\f}^{\circd{e}} \cdot \bv{v}_{\s}^{\circd{e}}
-
\tfrac{1}{2}\|\bv{v}_{\s}^{\circd{e}}\|^{2}
\rrbracket
\\
&=
d
\llbracket
\tfrac{1}{2}
\|\bv{v}_{\f}^{\circd{e}} - \bv{v}_{\s}^{\circd{e}}\|^{2}
+
\bv{v}_{\f}^{\circd{e}} \cdot \bv{v}_{\s}^{\circd{e}}
\rrbracket.
\end{aligned}
\end{equation}
Next, again recalling that $\bv{v}_{\s}^{\circd{e}}$ is continuous across $\mathcal{S}(t)$, using Eq.~\eqref{eq: Jump condition for mixture} we have that
\begin{equation}
\label{eq: power jump}
\llbracket
\ts{T}_{\s} \bv{m} \cdot \bv{v}_{\s}^{\circd{e}}
\rrbracket
=
d \llbracket\bv{v}_{\f}^{\circd{e}} \cdot \bv{v}_{\s}^{\circd{e}} \rrbracket
-
\llbracket
\ts{T}_{\f}\bv{m} \cdot \bv{v}_{\s}^{\circd{e}}
\rrbracket.
\end{equation}
Substituting Eqs.~\eqref{eq: jump kf} and~\eqref{eq: power jump} into Eq.~\eqref{eq: TPE almost done}, and recalling the definition in Eq.~\eqref{eq: power expended by the interface}, the claim follows.
\end{proof}

\begin{proposition}
    \label{claim: variation of relative velocity}
\begin{equation}
\label{eq: variation of relative velocity}
\frac{\partial(\bv{v}_{\f}^{\circd{s}} - \bv{v}_{\s})}{\partial \dot{q}} [\delta q] = 
-\ts{F}_{\s}\ts{F}_{\frs}^{-1} \, \delta \bv{\chi}_{\frs}.
\end{equation}
\end{proposition}
\begin{proof}
    Equation~\eqref{eq: variation of relative velocity} follows directly from Eq.~\eqref{eq: vf rel vs} and the definition of the list $q(t)$ in Eq.~\eqref{eq: hamiltons motion pair}.
\end{proof}

\begin{proposition}
\label{claim: hamiltons principle}
    Under the assumptions stated in the paper, Eq.~\eqref{eq: Hamiltons principle final global} holds.
\end{proposition}
\begin{proof}
The derivation presented herein makes repeated use of the standard identities such as those in~\eqref{eq: appx identity grad dot} and~\eqref{eq: appx identity div prod scalar vec}.

    The expression for the variation of the kinetic energy in this paper can be borrowed directly from Ref.~\refcite{dellIsola2009Boundary-Condit-0}. Adjusting for the differences in notation, this gives:
\begin{equation}
\label{eq: variation of K}
\begin{aligned}[b]
\int_{0}^{T}
\delta \mathscr{K} &= \int_{0}^{T}
\biggl\{
\int_{B_{\s}} -J_{\s}
\bigl(
\rho_{\s}^{\circd{s}} \bv{a}_{\s}
+
\rho_{\f}^{\circd{s}} \bv{a}_{\f}^{\circd{s}}
\bigr)
\cdot \delta \bv{\chi}_{\s} 
\\
&\qquad\qquad
-\int_{\mathcal{S}_{\s}}
\underbrace{%
\llbracket
\rho_{\f}^{\circd{s}} \bv{v}_{\f}^{\circd{s}} \otimes (\bv{v}_{\f}^{\circd{s}} - \bv{v}_{\s})
\rrbracket
}_{\llbracket
\rho_{\f}^{\circd{s}} (\bv{v}_{\f}^{\circd{s}} - \bv{v}_{\s}) \otimes (\bv{v}_{\f}^{\circd{s}} - \bv{v}_{\s})
\rrbracket}
J_{\s} \inversetranspose{\ts{F}}_{\s} \bv{m}_{\s} \cdot \delta\bv{\chi}_{\s}
\\
&\qquad\qquad
+
\int_{B_{\s}} J_{\s} \rho_{\f}^{\circd{s}} \bv{a}_{\f}^{\circd{s}} \cdot \ts{F}_{\s} \ts{F}_{\frs}^{-1} \delta \bv{\chi}_{\frs}
\\
&\qquad\qquad
-
\int_{\mathcal{S}_{\s}}
\Bigl\llbracket
\bigl(
\underbrace{%
\tfrac{1}{2} \rho_{\f}^{\circd{s}} \|\bv{v}_{\f}^{\circd{s}}\|^{2}}_{k_{\f}^{\circd{s}}} \, \ts{I}
\\
&\qquad\qquad\qquad\qquad
-
\rho_{\f}^{\circd{s}} \bv{v}_{\f}^{\circd{s}} \otimes (\bv{v}_{\f}^{\circd{s}} - \bv{v}_{\s})
\bigr) J_{\s} \inversetranspose{\ts{F}}_{\s} \bv{m}_{\s}
\cdot
\ts{F}_{\s}\ts{F}_{\frs}^{-1} \delta\bv{\chi}_{\frs}
\Bigr\rrbracket
\biggr\},
\end{aligned}
\end{equation}
where we have used the continuity of $\bv{v}_{\s}$ across $\mathcal{S}_{\s}$ to rewrite the term under bracket on the second line of the right-hand side.

To provide the form of the potential energy variation, we start with recalling that $\ts{C}_{\s} = \transpose{\ts{F}}_{\s} \ts{F}_{\s}$, with $\ts{F}_{\s} = \Grad{\bv{\chi}}_{\s}$. Hence, we have
\begin{equation}
\label{eq: variation of strain energy}
\delta\Psi_{\s} = \frac{\partial\Psi_{\s}}{\partial\ts{C}_{\s}} \colondot \delta \ts{C}_{\s}
~\Rightarrow~
\delta\Psi_{\s} = 2 \ts{F}_{\s} \frac{\partial\Psi_{\s}}{\partial\ts{C}_{\s}} \colondot \Grad{\delta\bv{\chi}_{\s}}
~\Rightarrow~
\delta\Psi_{\s} = \ts{P}_{\s} \colondot \Grad{\delta\bv{\chi}_{\s}},
\end{equation}
where we have taken advantage of the symmetry of $\ts{C}_{\s}$.

Recalling to Eqs.~\eqref{eq: us and uf defs},~\eqref{eq: uf from relative motion}, and~\eqref{eq: U definition}, and recalling that the terms $J_{\s}\bv{b}_{\s}$ and $J_{\s}\bv{b}_{\f}$ are treated as externally controlled entities at fixed $\bv{X}_{\s}$ and $\bv{X}_{\f}$, respectively, the variation of the potential energy is as follows:
\begin{equation}
\label{eq: Variation of U}
\begin{aligned}[b]
    \delta U &=
    \int_{B_{\s}}
    \bigl(
    \ts{P}_{\s}^{e} \colondot \Grad{\delta\bv{\chi}}_{\s}
    - J_{\s} (\bv{b}_{\s} + \bv{b}_{\f})^{\circd{s}} \cdot 
    \delta\bv{\chi}_{\s}
    + J_{\s} \bv{b}_{\f}^{\circd{s}} \cdot \ts{F}_{\s} \ts{F}_{\frs}^{-1} \delta\bv{\chi}_{\frs}
    \bigr)
\\
&= \int_{B_{\s}} 
\bigl(
\Div{[\transpose{(\ts{P}_{\s}^{e})} \delta\bv{\chi}_{\s}}] - 
(\Div{\ts{P}_{\s}^{e}} 
+ J_{\s} (\bv{b}_{\s} + \bv{b}_{\f})^{\circd{s}}) \cdot \delta\bv{\chi}_{\s}
    + J_{\s} \bv{b}_{\f}^{\circd{s}} \cdot \ts{F}_{\s} \ts{F}_{\frs}^{-1} \delta\bv{\chi}_{\frs}
    \bigr)
\\
&=
-\int_{\mathcal{S}_{\s}}
\llbracket
\ts{T}_{\s}^{e} J_{\s} \inversetranspose{\ts{F}}_{\s} \bv{m}_{\s} 
\cdot \delta\bv{\chi}_{\s} \rrbracket
\\
&\qquad\quad
-\int_{B_{\s}}
\bigl(
J_{\s}(\Div{\ts{T}_{\s}^{e}} 
+ \bv{b}_{\s} + \bv{b}_{\f})^{\circd{s}} \cdot \delta\bv{\chi}_{\s}
    - J_{\s} \bv{b}_{\f}^{\circd{s}} \cdot \ts{F}_{\s} \ts{F}_{\frs}^{-1} \delta\bv{\chi}_{\frs}
    \bigr).
\end{aligned}
\end{equation}

As $\mathscr{C}$ has been introduced in rate form, its variation is computed as $\delta\mathscr{C} = (\partial\mathscr{C}/\partial\dot{q})[\delta q]$. This gives 
\begin{equation}
\label{eq: variation of C}
\begin{aligned}[b]
    \delta \mathscr{C} &=
    \int_{B_{\s}}
    p^{\circd{s}} J_{\s} (\Div{(\delta\bv{\chi}_{\s}^{\circd{e}}}
    - \phi_{\f} (\ts{F}_{\s}\ts{F}_{\frs}^{-1} \delta\bv{\chi}_{\frs})^{\circd{e}}
    ))^{\circd{s}}
    \\
    &=
     \int_{B_{\s}}
     \bigl(
    \Div{(p^{\circd{s}} J_{\s} \ts{F}_{\s}^{-1}(\delta\bv{\chi}_{\s}
    - \phi_{\f}^{\circd{s}} \ts{F}_{\s}\ts{F}_{\frs}^{-1} \delta\bv{\chi}_{\frs}))}
    \\
    &\qquad\quad
    -\Div{(J_{\s}p^{\circd{s}}\inversetranspose{\ts{F}}_{\s})} \cdot (\delta \bv{\chi}_{\s}
     - \phi_{\f}^{\circd{s}} \ts{F}_{\s}\ts{F}_{\frs}^{-1} \delta\bv{\chi}_{\frs}
    )
    \bigr)
    \\
    &=
    -
    \int_{\mathcal{S}_{\s}}
    \llbracket
    (p^{\circd{s}} J_{\s} \ts{F}_{\s}^{-1}(\delta\bv{\chi}_{\s}
    - \phi_{\f}^{\circd{s}} \ts{F}_{\s}\ts{F}_{\frs}^{-1} \delta\bv{\chi}_{\frs}))
    \rrbracket \cdot \bv{m}_{\s}
     \\
    &\qquad\quad
    -
    \int_{B_{\s}} 
   \Div{(J_{\s}p^{\circd{s}}\inversetranspose{\ts{F}}_{\s})} \cdot \delta \bv{\chi}_{\s}
   \\
   &\qquad\quad
     + 
     \int_{B_{\s}}
     \underbrace{%
     \phi_{\f}^{\circd{s}} \Div{(J_{\s}p^{\circd{s}}\inversetranspose{\ts{F}}_{\s})}     }_{\Div{\,(J_{\s}\phi_{\f}^{\circd{s}}p^{\circd{s}}\inversetranspose{\ts{F}}_{\s})} - J_{\s}p^{\circd{s}}\inversetranspose{\ts{F}}_{\s}\Grad{\phi_{\f}^{\circd{s}}}}
     \cdot
     \ts{F}_{\s}\ts{F}_{\frs}^{-1} \delta\bv{\chi}_{\frs}
.
\end{aligned}
\end{equation}

Similarly, we have
\begin{equation}
\label{eq: variation of surface constraint}
\begin{aligned}[b]
    \delta\mathscr{C}_{\flt} &=
    \int_{\mathcal{S}_{\s}}
    -\ip^{\circd{s}}
    \llbracket
    \phi_{\f}^{\circd{s}}
    \ts{F}_{\s}\ts{F}_{\frs}^{-1} \delta\bv{\chi}_{\frs}
    \cdot
    J_{\s} \inversetranspose{\ts{F}}_{\s} \bv{m}_{\s}
    \rrbracket.
\end{aligned}
\end{equation}

The variation of $\mathscr{R}$, again computed by taking into account that the dissipation is written in rate form, is
\begin{equation}
\label{eq: variation of R step 1}
\begin{aligned}[b]
\delta \mathscr{R} &= 
\int_{B_{\s}}
J_{\s}
\Bigl(
2 \mu_{\f} \ts{D}_{\f} \colondot \Grad(\delta\bv{\chi}_{\s} - \ts{F}_{\s}\ts{F}_{\frs}^{-1} \delta\bv{\chi}_{\frs})^{\circd{e}}
\\
&\qquad\quad
-
2 \mu_{B} (\ts{D}_{\f} - \ts{D}_{\s}) \colondot \Grad(\ts{F}_{\s}\ts{F}_{\frs}^{-1}\delta\bv{\chi}_{\frs})^{\circd{e}}
\\
&\qquad\quad
-
\frac{\mu_{D} \phi_{\f}^{2}}{\kappa_{\s}}(\bv{v}_{\f}^{\circd{e}} - \bv{v}_{\s}^{\circd{e}}) \cdot 
(\ts{F}_{\s}\ts{F}_{\frs}^{-1} \delta\bv{\chi}_{\frs})^{\circd{e}}
\Bigr)^{\circd{s}}
\\
&\qquad
- 
\int_{\mathcal{S}_{\s}}
\tfrac{1}{2} \mu_{\mathcal{S}}
\llbracket
 \phi_{\f}^{\circd{s}} (\bv{v}_{\f}^{\circd{s}} - \bv{v}_{\s})
\rrbracket
\cdot
\llbracket
\phi_{\f}^{\circd{s}} \ts{F}_{\s}\ts{F}_{\frs}^{-1} \delta\bv{\chi}_{\frs} \rrbracket
J_{\s} \|\inversetranspose{\ts{F}}_{\s} \bv{m}_{\s} \|.
\end{aligned}
\end{equation}
The above expression can be rewritten as
\begin{equation}
\label{eq: variation of R step 2}
\begin{aligned}[b]
\delta \mathscr{R} &= 
\int_{B_{\s}}
J_{\s}
\Bigl(
2 \mu_{\f} \ts{D}_{\f} \colondot \Grad\delta\bv{\chi}_{\s}^{\circd{e}}
-  \ts{T}_{\f}^{v}
\colondot \Grad(\ts{F}_{\s}\ts{F}_{\frs}^{-1}\delta\bv{\chi}_{\frs})^{\circd{e}}
\\
&\qquad\quad
-
\frac{\mu_{D} \phi_{\f}^{2}}{\kappa_{\s}}(\bv{v}_{\f}^{\circd{e}} - \bv{v}_{\s}^{\circd{e}}) \cdot 
(\ts{F}_{\s}\ts{F}_{\frs}^{-1} \delta\bv{\chi}_{\frs})^{\circd{e}}
\Bigr)^{\circd{s}}
\\
&\quad
- 
\int_{\mathcal{S}_{\s}}
\bigl\llbracket
\tfrac{1}{2} \mu_{\mathcal{S}}
\llbracket
 \phi_{\f}^{\circd{s}} (\bv{v}_{\f}^{\circd{s}} - \bv{v}_{\s})
\rrbracket
\cdot
\phi_{\f}^{\circd{s}} \ts{F}_{\s}\ts{F}_{\frs}^{-1} \delta\bv{\chi}_{\frs} \bigr\rrbracket
J_{\s} \|\inversetranspose{\ts{F}}_{\s} \bv{m}_{\s} \|.
\end{aligned}
\end{equation}
Using the fact that for a field with domain in $B(t)$ like, say, $\ts{T}_{\f}$, we have
\begin{equation*}
J_{\s} (\Div{\ts{T}_{\f}})^{\circd{s}} = \Div{(J_{\s} \ts{T}_{\f}^{\circd{s}}\inversetranspose{\ts{F}}_{\s})},
\end{equation*}
we can further rewrite Eq.~\eqref{eq: variation of R step 2} as
\begin{equation}
\label{eq: variation of R step 3}
\begin{aligned}[b]
\delta \mathscr{R} &= 
\int_{B_{\s}}
\biggl[
\Div{
(\transpose{(2 \mu_{\f} J_{\s} \ts{D}_{\f}^{\circd{s}} \inversetranspose{\ts{F}}_{\s})}
\delta\bv{\chi}_{\s})
}
\\
&\qquad\quad
-  \Div{(\transpose{(J_{\s}(\ts{T}_{\f}^{v})^{\circd{s}}\inversetranspose{\ts{F}}
_{\s})} \ts{F}_{\s}\ts{F}_{\frs}^{-1}\delta\bv{\chi}_{\frs})
}
\\
&\qquad\quad
- \Div{
(2 \mu_{\f} J_{\s} \ts{D}_{\f}^{\circd{s}} \inversetranspose{\ts{F}}_{\s})}
 \cdot \delta\bv{\chi}_{\s}
 +
 \Div{(J_{\s}(\ts{T}_{\f}^{v})^{\circd{s}}\inversetranspose{\ts{F}}
_{\s})} \cdot \ts{F}_{\s}\ts{F}_{\frs}^{-1}\delta\bv{\chi}_{\frs}
\\
&\qquad\quad
-
J_{\s} \Bigl( \frac{\mu_{D} \phi_{\f}^{2}}{\kappa_{\s}}(\bv{v}_{\f}^{\circd{e}} - \bv{v}_{\s}^{\circd{e}}) \cdot 
(\ts{F}_{\s}\ts{F}_{\frs}^{-1} \delta\bv{\chi}_{\frs})^{\circd{e}}
\Bigr)^{\circd{s}}
\biggr]
\\
&\quad
- 
\int_{\mathcal{S}_{\s}}
\bigl\llbracket
\tfrac{1}{2} \mu_{\mathcal{S}}
\llbracket
 \phi_{\f}^{\circd{s}} (\bv{v}_{\f}^{\circd{s}} - \bv{v}_{\s})
\rrbracket
\cdot
\phi_{\f}^{\circd{s}} \ts{F}_{\s}\ts{F}_{\frs}^{-1} \delta\bv{\chi}_{\frs} \bigr\rrbracket
J_{\s} \|\inversetranspose{\ts{F}}_{\s} \bv{m}_{\s} \|.
\end{aligned}
\end{equation}
Applying integration by parts, we then obtain
\begin{equation}
\label{eq: variation of R step 4}
\begin{aligned}[b]
\delta \mathscr{R} &= 
-\int_{\mathcal{S}_{\s}}
\llbracket
2 \mu_{\f}  \ts{D}_{\f}^{\circd{s}}
\rrbracket
J_{\s}
\inversetranspose{\ts{F}}_{\s}
\bv{m}_{\s}
\cdot
\delta\bv{\chi}_{\s}
\\
&\qquad
+
\int_{\mathcal{S}_{\s}}
\llbracket
J_{\s}(\ts{T}_{\f}^{v})^{\circd{s}}
\inversetranspose{\ts{F}}
_{\s} \bv{m}_{\s}
\cdot
\ts{F}_{\s}\ts{F}_{\frs}^{-1}\delta\bv{\chi}_{\frs}
\rrbracket
\\
&\qquad
+
\int_{B_{\s}}
\biggl[
- \Div{
(2 \mu_{\f} J_{\s} \ts{D}_{\f}^{\circd{s}} \inversetranspose{\ts{F}}_{\s})}
 \cdot \delta\bv{\chi}_{\s}
 +
 \Div{(J_{\s}(\ts{T}_{\f}^{v})^{\circd{s}}\inversetranspose{\ts{F}}
_{\s})} \cdot \ts{F}_{\s}\ts{F}_{\frs}^{-1}\delta\bv{\chi}_{\frs}
\\
&\qquad\quad
-
J_{\s} \Bigl(\frac{\mu_{D} \phi_{\f}^{2}}{\kappa_{\s}}(\bv{v}_{\f}^{\circd{e}} - \bv{v}_{\s}^{\circd{e}}) \cdot 
(\ts{F}_{\s}\ts{F}_{\frs}^{-1} \delta\bv{\chi}_{\frs})^{\circd{e}}
\Bigr)^{\circd{s}}
\biggr]
\\
&\quad
- 
\int_{\mathcal{S}_{\s}}
\bigl\llbracket
\tfrac{1}{2} \mu_{\mathcal{S}}
\llbracket
 \phi_{\f}^{\circd{s}} (\bv{v}_{\f}^{\circd{s}} - \bv{v}_{\s})
\rrbracket
\cdot
\phi_{\f}^{\circd{s}} \ts{F}_{\s}\ts{F}_{\frs}^{-1} \delta\bv{\chi}_{\frs} \bigr\rrbracket
J_{\s} \|\inversetranspose{\ts{F}}_{\s} \bv{m}_{\s} \|.
\end{aligned}
\end{equation}

The claim follows from substituting the contributions from Eqs.~\eqref{eq: variation of K}, \eqref{eq: Variation of U}--\eqref{eq: variation of surface constraint}, and Eq.~\eqref{eq: variation of R step 4} into Eq.~\eqref{eq: Hamiltons principle paradigm}, applying a localization argument in time, and recalling the definitions for $\ts{T}_{\s}$, $\ts{T}_{\f}$, and $\bv{p}_{\f}$. 
\end{proof}

\section*{Acknowledgments}
The authors gratefully acknowledge partial support from the Pennsylvania Department of Health using Tobacco CURE Funds (the Department specifically disclaims responsibility for any analyses, interpretations, or conclusions).

The authors also wish to express their gratitude to Profs.~Bruce J.~Gluckman and Patrick J.~Drew, both of the Penn State Center for Neural Engineering, and to Prof.~Anna Mazzucato of the Penn State Department of Mathematics for their feedback and suggestions.

} 
\bibliographystyle{ws-m3as} 
\bibliography{2025CostanzoEtAlJumpConditionsPoroelsticMedia_arXiv}

\begin{thebibliography}{10}
\newcommand{\enquote}[1]{#1}

\bibitem{Abbott2004Evidence-for-bu}
N.~J. Abbott, \enquote{Evidence for bulk flow of brain interstitial fluid:
  significance for physiology and pathology}, {\it Neurochemistry
  International} \textbf{45} (2004) 545--52, {DOI:
  10.1016/j.neuint.2003.11.006; PMID: 15186921}.

\bibitem{Abbott2018The-role-of-bra}
N.~J. Abbott, M.~E. Pizzo, J.~E. Preston, D.~Janigro and R.~G. Thorne,
  \enquote{The role of brain barriers in fluid movement in the {CNS}: is there
  a 'glymphatic' system?}, {\it Acta Neuropathologica} \textbf{135} (2018)
  387--407, {DOI: 10.1007/s00401-018-1812-4; PMID: 29428972}.

\bibitem{Bansal2024A-Lagrange-Mult}
A.~Bansal, N.~A. Barnafi, D.~N. Pandey and R.~Ruiz-Baier, \enquote{A {L}agrange
  multiplier-based method for {S}tokes-linearized poro-hyperelastic interface
  problems}, {\it arXiv} {arXiv link:
  https://doi.org/10.48550/arXiv.2407.13684}.

\bibitem{Barnafi2021Mathematical-An}
N.~Barnafi, P.~Zunino, L.~Ded{\`e} and A.~Quarteroni, \enquote{Mathematical
  analysis and numerical approximation of a general linearized
  poro-hyperelastic model}, {\it Computers and Mathematics with Applications}
  \textbf{91} (2021) 202--228, {DOI: 10.1016/j.camwa.2020.07.025}.

\bibitem{Beavers1967Boundary-Condi}
G.~S. Beavers and D.~D. Joseph, \enquote{Boundary conditions at a naturally
  permeable wall}, {\it Journal of Fluid Mechanics} \textbf{30} (1967)
  197--–207, {DOI: 10.1017/S0022112067001375}.

\bibitem{Bedford2021Hamiltons-Princ}
A.~Bedford, {\it Hamilton's Principle in Continuumm Mechanics} (Springer Cham,
  2021).

\bibitem{Bedussi2018Paravascular-sp}
B.~Bedussi, M.~Almasian, J.~de~Vos, E.~VanBavel and E.~N. Bakker,
  \enquote{Paravascular spaces at the brain surface: {L}ow resistance pathways
  for cerebrospinal fluid flow}, {\it Journal of Cerebral Blood Flow \&
  Metabolism} \textbf{38} (2018) 719--726, {DOI: 10.1177/0271678X17737984;
  PMID: 29039724; PMCID: PMC5888857}.

\bibitem{Bowen1976Theory-of-Mixtu0}
R.~M. Bowen, \enquote{Theory of mixtures}, in {\it Continuum Physics}, ed.
  A.~C. Eringen (Academic Press, New York, 1976), volume III---Mixtures and EM
  Field Theories, pp. 1--127.

\bibitem{Bowen1980Incompressible-0}
R.~M. Bowen, \enquote{Incompressible porous media models by use of the theory
  of mixtures}, {\it International Journal of Engineering Science} \textbf{18}
  (1980) 1129--1148, {DOI: 10.1016/0020-7225(80)90114-7}.

\bibitem{Bradbury1981Drainage-of-cer}
M.~W. Bradbury, H.~F. Cserr and R.~J. Westrop, \enquote{Drainage of cerebral
  interstitial fluid into deep cervical lymph of the rabbit}, {\it American
  Journal of Physiology--Renal Physiology} \textbf{240} (1981) F329--36, {DOI:
  10.1152/ajprenal.1981.240.4.F329; PMID: 7223890}.

\bibitem{Bradbury1983Factors-influen}
M.~W. Bradbury and R.~J. Westrop, \enquote{Factors influencing exit of
  substances from cerebrospinal fluid into deep cervical lymph of the rabbit},
  {\it The Journal of Physiology} \textbf{339} (1983) 519--34, {DOI:
  10.1113/jphysiol.1983.sp014731; PMID: 6411905; PMCID: PMC1199176}.

\bibitem{Brezzi1991Mixed-and-Hybrid-0}
F.~Brezzi and M.~Fortin, {\it Mixed and Hybrid Finite Element Methods},
  volume~15 of {\it Springer Series in Computational Mathematics}
  (Springer-Verlag, New York, 1991).

\bibitem{Brinkman1949A-Calculation-o}
H.~C. Brinkman, \enquote{A calculation of the viscous force exerted by a
  flowing fluid on a dense swarm of particles}, {\it Applied Scientific
  Research} \textbf{1} (1949) 27--34, {DOI: 10.1007/BF02120313}.

\bibitem{Causemann2022Human-Intracran}
M.~Causemann, V.~Vinje and M.~E. Rognes, \enquote{Human intracranial
  pulsatility during the cardiac cycle: {A} computational modelling framework},
  {\it Fluids and Barriers of the {CNS}} \textbf{19} (2022) 84--1--84--17,
  {DOI: 10.1186/s12987-022-00376-2; PMID: 36320038; PMCID: PMC9623946}.

\bibitem{COMSOLCite}
{COMSOL AB}, {\it {COMSOL} {M}ultiphysics\textsuperscript{\textregistered}
  v.~6.2} (Stockholm, Sweden, 2024), \url{www.comsol.com}.

\bibitem{Costanzo2016Finite-Element-0}
F.~Costanzo and S.~T. Miller, \enquote{An arbitrary {L}agrangian--{E}ulerian
  finite element formulation for a poroelasticity problem stemming from mixture
  theory}, {\it Computer Methods in Applied Mechanics and Engineering}
  \textbf{323} (2017) 64--97, {DOI: 10.1016/j.cma.2017.05.006}.

\bibitem{Darcy1856Les-Fontaines-P}
H.~Darcy, {\it Les Fontaines Publiques De La Ville De {D}ijon} (Victor Dalmont,
  Libraire des Corps Imperiaux des Ponts et Chauss{\'e}es et des Mines, Paris,
  1856).

\bibitem{dellIsola2009Boundary-Condit-0}
F.~dell'Isola, A.~Madeo and P.~Seppecher, \enquote{Boundary conditions at
  fluid-permeable interfaces in porous media: A variational approach}, {\it
  International Journal of Solids and Structures} \textbf{46} (2009)
  3150--3164, {DOI: 10.1016/j.ijsolstr.2009.04.008}.

\bibitem{Discacciati2009Navier-Sto}
M.~Discacciati and A.~Quarteroni, \enquote{Navier-stokes/darcy coupling:
  Modeling, analysis, and numerical approximation}, {\it Revista Matemática
  Complutense} \textbf{22} (2009) 315--426, {DOI:
  10.5209/rev\_REMA.2009.v22.n2.16263}.

\bibitem{Fernandez2009The-Derivation-0}
M.~A. Fern{\'a}ndez, L.~Formaggia, J.-F. Gerbeau and A.~Quarteroni,
  \enquote{The derivation of the equations for fluids and structure}, in {\it
  Cardiovascular Mathematics: Modeling and Simulation of the Circulatory
  System}, eds. L.~Formaggia, A.~Quarteroni and A.~Veneziani (Springer, 2009),
  volume~1 of {\it Modeling, Simulation and Applications (MS\&A)}, ch.~3, pp.
  77--121, {DOI: 10.1007/978-88-470-1152-6\_3}.

\bibitem{Franca1992Stabilized-Finite-Element-1}
L.~P. Franca and S.~L. Frey, \enquote{Stabilized finite-element methods .{II}.
  {T}he incompressible {N}avier-{S}tokes equations}, {\it Computer Methods in
  Applied Mechanics and Engineering} \textbf{99} (1992) 209--233, {DOI:
  10.1016/0045-7825(92)90041-H}.

\bibitem{Gurtin-CMBook-1981-1}
M.~E. Gurtin, {\it An Introduction to Continuum Mechanics}, volume 158 of {\it
  Mathematics in Science and Engineering} (Academic Press, San Diego, 1981).

\bibitem{Gurtin2000Configurational0}
M.~E. Gurtin, {\it Configurational Forces as Basic Concepts of Continuum
  Physics}, volume 137 of {\it Applied Mathematical Sciences} (Springer, New
  York, 2000).

\bibitem{GurtinFried_2010_The-Mechanics_0}
M.~E. Gurtin, E.~Fried and L.~Anand, {\it The Mechanics and Thermodynamics of
  Continua} (Cambridge University Press, New York, 2010).

\bibitem{Hladky2018Elimination-of-}
S.~B. Hladky and M.~A. Barrand, \enquote{Elimination of substances from the
  brain parenchyma: efflux via perivascular pathways and via the blood--brain
  barrier}, {\it Fluids and Barriers of the {CNS}} \textbf{15} (2018) 30, {DOI:
  10.1186/s12987-018-0113-6; PMID: 30340614; PMCID: PMC6194691}.

\bibitem{Hladky2022The-glymphatic-}
S.~B. Hladky and M.~A. Barrand, \enquote{The glymphatic hypothesis: the theory
  and the evidence}, {\it Fluids and Barriers of the {CNS}} \textbf{19} (2022)
  9, {DOI: 10.1186/s12987-021-00282-z; PMID: 35115036; PMCID: PMC8815211}.

\bibitem{Hou1989Boundary-condit}
J.~S. Hou, M.~H. Holmes, W.~M. Lai and V.~C. Mow, \enquote{Boundary conditions
  at the cartilage--synovial fluid interface for joint lubrication and
  theoretical verifications}, {\it Journal of Biomechanical Engineering}
  \textbf{111} (1989) 78--87, {DOI: https://doi.org/10.1115/1.3168343; PMID:
  2747237}.

\bibitem{Hu2017A-Nonconforming}
X.~Hu, C.~Rodrigo, F.~J. Gaspar and L.~T. Zikatanov, \enquote{A nonconforming
  finite element method for the {B}iot's consolidation model in
  poroelasticity}, {\it Journal of Computational and Applied Mathematics}
  \textbf{310} (2017) 143--154, {DOI: 10.1016/j.cam.2016.06.003}.

\bibitem{Iliff2013Brain--wide-pat}
J.~J. Iliff, H.~Lee, M.~Yu, T.~Feng, J.~Logan, M.~Nedergaard and H.~Benveniste,
  \enquote{Brain--wide pathway for waste clearance captured by
  contrast--enhanced {MRI}}, {\it The Journal of Clinical Investigation}
  \textbf{123} (2013) 1299--309, {DOI: 10.1172/JCI67677; PMID: 23434588; PMCID:
  PMC3582150}.

\bibitem{Iliff2019The-glymphatic-}
J.~J. Iliff and M.~Simon, \enquote{The glymphatic system supports convective
  exchange of cerebrospinal fluid and brain interstitial fluid that is mediated
  by perivascular aquaporin--4}, {\it The Journal of Physiology} \textbf{597}
  (2019) 4417--4419, {DOI: 10.1113/JP277635; PMID: 31389028; PMCID:
  PMC7236551}.

\bibitem{Iliff2012A-Paravascular-}
J.~J. Iliff, M.~Wang, Y.~Liao, B.~A. Plogg, W.~Peng, G.~A. Gundersen,
  H.~Benveniste, G.~Vates, R.~Deane, S.~A. Goldman, E.~A. Nagelhus and
  M.~Nedergaard, \enquote{A {P}aravascular {P}athway {F}acilitates {CSF} {F}low
  {T}hrough the {B}rain {P}arenchyma and the {C}learance of {I}nterstitial
  {S}olutes, {I}ncluding {A}myloid $\beta$}, {\it Science Translational
  Medicine} \textbf{4} (2012) 147ra111, {DOI: 10.1126/scitranslmed.3003748;
  PMID: 22896675; PMCID: PMC3551275}.

\bibitem{Iliff2013Cerebral-arteri}
J.~J. Iliff, M.~Wang, D.~M. Zeppenfeld, A.~Venkataraman, B.~A. Plog, Y.~Liao,
  R.~Deane and M.~Nedergaard, \enquote{Cerebral arterial pulsation drives
  paravascular {CSF}--interstitial fluid exchange in the murine brain}, {\it
  The Journal of Neuroscience} \textbf{33} (2013) 18190--9, {DOI:
  10.1523/JNEUROSCI.1592-13.2013; PMID: 24227727; PMCID: PMC3866416}.

\bibitem{Kedarasetti2022Arterial-vasodi}
R.~T. Kedarasetti, P.~J. Drew and F.~Costanzo, \enquote{Arterial vasodilation
  drives convective fluid flow in the brain: a poroelastic model}, {\it Fluids
  and Barriers of the {CNS}} \textbf{19} (2022) 34, {DOI:
  https://doi.org/10.1186/s12987-022-00326-y; PMID: 35570287; PMCID:
  PMC9107702}.

\bibitem{Kedarasetti2020Functional-hype}
R.~T. Kedarasetti, K.~L. Turner, C.~Echagarruga, B.~J. Gluckman, P.~J. Drew and
  F.~Costanzo, \enquote{Functional hyperemia drives fluid exchange in the
  paravascular space}, {\it Fluids and Barriers of the {CNS}} \textbf{17}
  (2020) 52, {DOI: 10.1186/s12987-020-00214-3; PMID: 32819402; PMCID:
  PMC7441569}.

\bibitem{Kiviniemi2016Ultra-fast-magn}
V.~Kiviniemi, X.~Wang, V.~Korhonen, T.~Kein{\"a}nen, T.~Tuovinen, J.~Autio,
  P.~LeVan, S.~Keilholz, Y.~F. Zang, J.~Hennig and M.~Nedergaard,
  \enquote{Ultra-fast magnetic resonance encephalography of physiological brain
  activity -- glymphatic pulsation mechanisms?}, {\it Journal of Cerebral Blood
  Flow \& Metabolism} \textbf{36} (2016) 1033--45, {DOI:
  10.1177/0271678X15622047; PMID: 26690495; PMCID: PMC4908626}.

\bibitem{Lee2015The-Effect-of-B}
H.~Lee, L.~Xie, M.~Yu, H.~Kang, T.~Feng, R.~Deane, J.~Logan, M.~Nedergaard and
  H.~Benveniste, \enquote{The {E}ffect of {B}ody {P}osture on {B}rain
  {G}lymphatic {T}ransport}, {\it The Journal of Neuroscience} \textbf{35}
  (2015) 11034--44, {DOI: 10.1523/JNEUROSCI.1625-15.2015; PMID: 26245965;
  PMCID: PMC4524974}.

\bibitem{Lee2019A-Mixed-Finite-}
J.~J. Lee, E.~Piersanti, K.~A. Mardal and M.~E. Rognes, \enquote{A mixed finite
  element method for nearly incompressible multiple-network poroelasticity},
  {\it {SIAM} Journal on Scientific Computing} \textbf{41} (2019) A722--A747,
  {DOI: 10.1137/18M1182395}.

\bibitem{Masud2002A-Stabilized-Mixed-0}
A.~Masud and T.~J.~R. Hughes, \enquote{A stabilized mixed finite element method
  for {D}arcy flow}, {\it Computer Methods in Applied Mechanics and
  Engineering} \textbf{191} (2002) 4341--4370.

\bibitem{Masud2006A-Multiscale-Fi}
A.~Masud and R.~A. Khurram, \enquote{A multiscale finite element method for the
  incompressible {N}avier-{S}tokes equations}, {\it Computer Methods in Applied
  Mechanics and Engineering} \textbf{195} (2006) 1750--1777, {DOI:
  10.1016/j.cma.2005.05.048}.

\bibitem{Mestre2018Flow-of-cerebro}
H.~Mestre, J.~Tithof, T.~Du, W.~Song, W.~Peng, A.~M. Sweeney, G.~Olveda, J.~H.
  Thomas, M.~Nedergaard and D.~H. Kelley, \enquote{Flow of cerebrospinal fluid
  is driven by arterial pulsations and is reduced in hypertension}, {\it Nature
  Communications} \textbf{9} (2018) 4878, {DOI: 10.1038/s41467-018-07318-3;
  PMID: 30451853; PMCID: PMC6242982}.

\bibitem{Nedergaard2013Garbage-Truck-o}
M.~Nedergaard, \enquote{Garbage {T}ruck of the {B}rain}, {\it Science}
  \textbf{340} (2013) 1529--30, {DOI: 10.1126/science.1240514; PMID: 23812703;
  PMCID: PMC3749839}.

\bibitem{Quarteroni2000Numerical-Mathematics-0}
A.~Quarteroni, R.~Sacco and F.~Saleri, {\it Numerical Mathematics} (Springer,
  Berlin, Heidelberg, New York, 2007), 2nd edition, {DOI: 10.1007/b98885}.

\bibitem{Rasmussen2022Fluid-transport}
M.~Rasmussen, H.~Mestre and M.~Nedergaard, \enquote{Fluid transport in the
  brain}, {\it Physiological Reviews} \textbf{102} (2022) 1025--1151, {DOI:
  10.1152/physrev.00031.2020; PMID: 33949874; PMCID: PMC8897154}.

\bibitem{Rodrigo2018New-Stabilized-}
C.~Rodrigo, X.~Hu, P.~Ohm, J.~H. Adler, F.~J. Gaspar and L.~T. Zikatanov,
  \enquote{New stabilized discretizations for poroelasticity and the {S}tokes'
  equations}, {\it Computer Methods in Applied Mechanics and Engineering}
  \textbf{341} (2018) 467--484, {DOI: 10.1016/j.cma.2018.07.003}.

\bibitem{Ruiz-Baier2022The-Biot-Stokes}
R.~Ruiz-Baier, M.~Taffetani, H.~D. Westermeyer and I.~Yotov, \enquote{The
  biot-stokes coupling using total pressure: {F}ormulation, analysis and
  application to interfacial flow in the eye}, {\it Computer Methods in Applied
  Mechanics and Engineering} \textbf{389} (2022) 114384--1--114384--30, {DOI:
  10.1016/j.cma.2021.114384}.

\bibitem{Salari2000Code-Verification-0}
K.~Salari and P.~Knupp, \enquote{Code verification by the method of
  manufactured solutions}, Sand2000-1444, Sandia National Laboratories, 2000.

\bibitem{Selkoe2016The-amyloid-hyp}
D.~J. Selkoe and J.~Hardy, \enquote{The amyloid hypothesis of {A}lzheimer's
  disease at 25 years}, {\it {EMBO} Molecular Medicine} \textbf{8} (2016)
  595--608, {DOI: 10.15252/emmm.201606210; PMID: 27025652; PMCID: PMC4888851}.

\bibitem{Shim2022A-Hybrid-Biphas}
J.~J. Shim and G.~A. Ateshian, \enquote{A hybrid biphasic mixture formulation
  for modeling dynamics in porous deformable biological tissues}, {\it Archive
  of Applied Mechanics} \textbf{92} (2022) 491--511, {DOI:
  10.1007/s00419-020-01851-8; PMID: 35330673; PMCID: PMC8939891}.

\bibitem{Sykova2004Diffusion-prope}
E.~Sykov{\'a}, \enquote{Diffusion properties of the brain in health and
  disease}, {\it Neurochemistry International} \textbf{45} (2004) 453--66,
  {DOI: 10.1016/j.neuint.2003.11.009; PMID: 15186911}.

\bibitem{Sykova2008Diffusion-in-br}
E.~Sykov{\'a} and C.~Nicholson, \enquote{Diffusion in brain extracellular
  space}, {\it Physiological Reviews} \textbf{88} (2008) 1277--340, {DOI:
  10.1152/physrev.00027.2007; PMID: 18923183; PMCID: PMC2785730}.

\bibitem{Truesdell1965The-Non-Linear-0}
C.~Truesdell and W.~Noll, \enquote{The non-linear field theories of mechanics},
  in {\it Handbuch der Physik}, ed. S.~Fl{\"u}gge (Springer-Verlag, 1965),
  volume III/3.

\bibitem{Vargova2011Glia-and-extrac}
l.~Vargova and E.~Sykov{\'a}, \enquote{Glia and extracellular matrix changes
  affect extracellular diffusion and volume transmission in the brain in health
  and disease}, {\it Glia} \textbf{59} (2011) S38--S39, {DOI:
  https://doi.org/10.1002/glia.21209}.

\bibitem{Verkman2013Diffusion-in-th}
A.~S. Verkman, \enquote{Diffusion in the extracellular space in brain and
  tumors}, {\it Physical Biology} \textbf{10} (2013) 045003, {DOI:
  10.1088/1478-3975/10/4/045003; PMID: 23913007; PMCID: PMC3937300}.

\bibitem{Wang2007New-Finite-Elem}
J.~Wang and X.~Ye, \enquote{New finite element methods in computational fluid
  dynamics by ${H}(\rm div)$ elements}, {\it {SIAM} Journal on Numerical
  Analysis} \textbf{45} (2007) 1269--1286, {DOI: 10.1137/060649227}.

\bibitem{Ye2021A-Stabilizer-Fr}
X.~Ye and S.~Zhang, \enquote{A stabilizer-free pressure-robust finite element
  method for the {S}tokes equations}, {\it Advances in Computational
  Mathematics} \textbf{47} (2021) 28--1--28--17, {DOI:
  10.1007/s10444-021-09856-9}.

\bibitem{Zhang2024BDM-Hrm-div-Wea}
S.~Zhang and P.~Zhu, \enquote{{BDM} ${H}(\rm div)$ weak {G}alerkin finite
  element methods for {S}tokes equations}, {\it Applied Numerical Mathematics}
  \textbf{197} (2024) 307--321, {DOI: 10.1016/j.apnum.2023.11.021}.

\end{thebibliography}
\end{document}